\documentclass[12pt]{amsart}
\usepackage[utf8]{inputenc}

\usepackage[margin=1.3in]{geometry}

\usepackage{amsmath}
\usepackage{amsfonts}
\usepackage{amssymb}
\usepackage{amsthm}
\usepackage{mathtools}
\usepackage{caption}
\usepackage{subcaption}
\usepackage{bbm}
\usepackage[export]{adjustbox}

\usepackage{stmaryrd}

\usepackage[all]{xy}

\usepackage{tikz-cd}
\usetikzlibrary{matrix}
\usepackage{graphicx} 
\usepackage{float}

\usepackage{epstopdf}

\usepackage{overpic}

\usepackage[linktocpage]{hyperref}
\hypersetup{
    colorlinks=true,
    linkcolor=blue,
    citecolor=blue,      
    urlcolor=blue,
}

\usepackage{color}
\definecolor{note}{rgb}{0,0,1}  

\newtheorem{theorem}{Theorem}

\newtheorem{proposition}[theorem]{Proposition}
\newtheorem*{proposition*}{Proposition}
\newtheorem{lemma}[theorem]{Lemma}

\newtheorem{corollary}[theorem]{Corollary}
\newtheorem{example}[theorem]{Example}

\numberwithin{equation}{section}
\numberwithin{theorem}{section}

\usepackage{enumitem}

\usepackage{todonotes}

\newcommand{\op}{\operatorname}

\newcommand{\R}{\mathbb{R}}

\newcommand{\Z}{\mathbb{Z}}

\newcommand{\be}{\begin{enumerate}}
\newcommand{\ee}{\end{enumerate}}

\usepackage[english]{babel}
\usepackage{csquotes}

\usepackage{hyphenat}

\usepackage[backend=biber,style=alphabetic,maxalphanames=4,maxnames=4]{biblatex}

\renewbibmacro{in:}{}

\DeclareDelimFormat[bib,biblist]{nametitledelim}{\addcomma\space}

\DeclareFieldFormat*{title}{\mkbibitalic{#1}\addcomma}
\DeclareFieldFormat*{journaltitle}{#1}
\DeclareFieldFormat*{volume}{\mkbibbold{#1}}
\DeclareFieldFormat{pages}{#1}
\DeclareFieldFormat[misc]{date}{preprint {#1}}
\DeclareFieldFormat{mr}{%
  MR\addcolon\space
  \ifhyperref
    {\href{http://www.ams.org/mathscinet-getitem?mr=MR#1}{\nolinkurl{#1}}}
    {\nolinkurl{#1}}}
    
\AtEveryBibitem{
  \clearfield{url}
  \clearfield{issn}
  \clearfield{isbn}
  \clearfield{eprintclass}
}
\AtEveryBibitem{\ifentrytype{book}{\clearfield{pages}}{}}

\usepackage{fancyhdr}
\newcommand{\one}{\mathbf{1}}

\DeclareMathOperator{\vol}{vol}
\DeclareMathOperator{\MV}{MV}

\DeclareMathOperator{\Ext}{Ext}
\DeclareMathOperator{\ext}{ext}
\DeclareMathOperator{\ME}{ME}

\newcommand{\red}{\mathrm{red}}
\newcommand{\eqdot}{\mathrel{\doteq}}
\allowdisplaybreaks[1]
\title{Log-concavity of flat arrangement polynomials}

\author{Yuan Gao}
\address{School of Mathematics, Nanjing University, Nanjing, Jiangsu, 210093, China}
\email{yuangao@nju.edu.cn} \urladdr{}

\author{Tianyu Yuan}
\address{School of Mathematical Sciences, Eastern Institute of Technology, Ningbo, Zhejiang, 315200, China}
\email{tyyuan@eitech.edu.cn} \urladdr{}

\date{\today}


\begin{document}

\begin{abstract}
    We prove log-concavity for the determinant-weighted external semi-activity polynomials of all real flat arrangements, strengthening their known trapezoidality. In fact, we establish a quadratic coefficient inequality that, in rank at least two, implies power concavity with an explicit rank-dependent exponent. The proof uses a new mixed-volume representation of the coefficients and the Alexandrov--Fenchel inequality.
    A more general formula gives a factorization and log-concavity for related mixed-volume sequences.
    As applications, we establish the conjectured log-concavity for spanning-tree polynomials of Eulerian digraphs, extend it to positive circulation weights, and strengthen the coefficient inequalities for Alexander polynomials of special alternating links.
\end{abstract}

\maketitle

\tableofcontents

\section{Introduction}\label{section-introduction}

Coefficient inequalities connect enumerative combinatorics with the geometry of convex bodies. In the setting of Alexander polynomials, they also connect graph enumeration with knot theory.
Fox's trapezoidality conjecture for alternating knots and its usual extension to nonsplit alternating links predicts that the absolute coefficients of the Alexander polynomial increase strictly to a possibly nontrivial plateau and then decrease strictly \cite{Fox62}\cite[Conjecture 1.1]{HMV25}.
Hafner, M\'esz\'aros, and Vidinas proved the stronger conclusion of log-concavity for special alternating links \cite[Theorem 1.2]{HMV24}.
Murasugi and Stoimenow's spanning-tree polynomial extends the special-alternating Alexander polynomial model to connected Eulerian digraphs \cite{MS03}.
Hafner, M\'esz\'aros, and Vidinas conjectured log-concavity, together with the absence of internal zeros, for this entire graph family \cite[Conjecture 1.6]{HMV25}.

Flat-arrangement polynomials provide a common setting for these questions. They enumerate bases of a real vector arrangement by external semi-activity, with each basis weighted by its determinant.
The definition and its invariance under a generic auxiliary choice originate in work of Li and Postnikov, as recorded by K\'alm\'an, M\'esz\'aros, and Postnikov \cite[Definition 3.3 and Theorem 3.4]{KMP25}. The latter authors developed the cographic realization of the Eulerian-digraph polynomial and gave an alternative proof of trapezoidality in the special-alternating setting.
In our earlier paper \cite[Theorem 1.2 and Corollary 6.4]{GY26}, we proved trapezoidality for every flat arrangement and every connected Eulerian digraph.
In this paper, we pursue the answer to the natural question whether all flat-arrangement polynomials satisfy the stronger form of trapezoidality, namely log-concavity.

Our proof for trapezoidality expresses the polynomial of a flat arrangement as a positive integral of the ordinary $h$-polynomials of generic fibers of a simplex projection \cite[Proposition 5.2]{GY26}.
We reply on the Dehn--Sommerville symmetry and the $g$-theorem to get the trapezoidal coefficient shape of those fibers and take the integral representation of them. However, this approach does not establish log-concavity, as we observed in \cite[Remark 6.5]{GY26}.
This led us to seek a coefficient formula where well-established log-concave inequalities can directly apply. 

We observe that the Alexandrov--Fenchel inequality is one of the central inequalities in convex geometry that provide log-concavity. To apply it, we need a mix volume expression for the flat arrangement polynomial coefficients.
The new input in the current paper is that we consider the images of all hypersimplices under the same projection. Their consecutive mixed volumes recover the coefficients individually, so Alexandrov--Fenchel can be applied.
The resulting inequality obtained via this approach turns out to be stronger than ordinary log-concavity.

\subsection{The main coefficient inequality}

Let $V$ be a real vector space of dimension $d\geq1$, equipped with a nonzero alternating volume form $\omega_V \in\bigwedge^dV^*$.
A \emph{flat arrangement} in $V$ is a labeled spanning family of vectors, represented by a surjective map
\begin{equation}
\label{eq-arrangement}
    A=[a_1\,|\cdots |\, a_N]\colon\R^N\longrightarrow V,
\end{equation}
such that there is a linear functional $\varphi\in V^*$ satisfying $\varphi(a_j)=1$ for every $j$.
The columns lie in the affine hyperplane $H=\varphi^{-1}(1)$, whose direction space is $W=\ker\varphi$.
Throughout the paper, we denote $V$ the column space and $W$ this fixed codimension-one subspace. We have
\begin{equation}\label{eq-dimensions}
    \dim V=d,\qquad\dim H=\dim W=d-1.
\end{equation}
When $d\geq2$, we write $n=d-1$ for the dimension in which the projected hypersimplex mixed volumes are computed.
Write $[N]=\{1,\ldots,N\}$ and $r=N-d$.

A \emph{basis} is a subset $B\subset[N]$ of size $d$ for which $A_B\coloneqq A|_{\R^B}\colon\R^B\to V$ is an isomorphism.
Its weight is
\begin{equation}\label{eq-basis-weight}
    w_A(B)=\vol_V\left(\sum_{b\in B}[0, a_b]\right)=|\omega_V(\{a_b,\,b\in B\})|.
\end{equation}
The sum here is the Minkowski sum, and the absolute value makes the ordering in the last expression irrelevant.
We denote the set of basis by $\mathcal{B}$.
A \emph{circuit vector} is a nonzero vector of $\ker A$ with inclusion-minimal support.
For $j\notin B$, the normalized fundamental circuit vector $c^{B,j}$ is defined by
\begin{equation}\label{eq-fundamental}
    c^{B,j}_j=1,\quad c^{B,j}_B=-A_B^{-1}a_j,\quad c^{B,j}_k=0\quad (k\notin B\sqcup\{j\}),
\end{equation}
whose support is the unique circuit contained in $B\sqcup\{j\}$.

We say $\rho\in\R^N$ is a \emph{circuit-generic} vector if $\langle\rho,z\rangle\neq0$ for every circuit vector $z$.
Such vectors exist because there are only finitely many circuit supports, each carrying a one-dimensional dependence space. For a circuit-generic vector $\rho$, we define
\begin{equation}\label{eq-activity}
    \Ext_\rho(B)=\{j\notin B\mid\langle\rho,c^{B,j}\rangle>0\},\quad \ext_\rho(B)=|\Ext_\rho(B)|.
\end{equation}
This is the external semi-activity convention of \cite[Definitions 3.1 and 3.3]{KMP25}. The flat-arrangement polynomial is
\begin{equation}\label{eq-polynomial}
    f_{A,\rho}(t)\coloneqq\sum_{B\in \mathcal{B}} w_A(B)t^{\ext_\rho(B)}.
\end{equation}
It is independent of $\rho$ by \cite[Theorem 3.4]{KMP25}, hence we write
\begin{equation}\label{eq-coefficients}
    f_A(t)=f_{A,\rho}(t)=\sum_{i=0}^r\alpha_i t^i,\quad \alpha_i=0\quad\text{for }i\notin\{0,\ldots,r\}.
\end{equation}

\begin{theorem}
    \label{thm-main}
    The coefficients $\alpha_0,\ldots,\alpha_r$ of \eqref{eq-coefficients} are positive and palindromic:
    \begin{equation}\label{eq-palindromicity}
        \alpha_i=\alpha_{r-i},\qquad 0\leq i\leq r.
    \end{equation}
    With the zero extension in \eqref{eq-coefficients}, we have
    \begin{equation}\label{eq-quadratic}
        d\alpha_i^2\geq \alpha_i(\alpha_{i-1}+\alpha_{i+1})+(d-2)\alpha_{i-1}\alpha_{i+1},\qquad 0\leq i\leq r.
    \end{equation}
    In particular, the coefficient sequence is log-concave. At an internal index
    \begin{equation}\label{eq-lc-equality}
        \alpha_i^2=\alpha_{i-1}\alpha_{i+1}\quad \Longleftrightarrow \quad \alpha_{i-1}=\alpha_i=\alpha_{i+1}.
    \end{equation}
    Therefore, the coefficient sequence is trapezoidal with one centered plateau, and $\deg f_A=r$.
    For $d=1$, all coefficients are equal.
\end{theorem}

Here a positive sequence $(u_i)$ is \emph{log-concave} if $u_i^2\geq u_{i-1}u_{i+1}$ at every internal index. The endpoint instances of \eqref{eq-quadratic} give $\alpha_1\leq d\alpha_0$ and its reflected counterpart when $r\geq1$.
They will lead to a binomial upper bound on the growth of all coefficients.
The rank-sensitive inequality also has the following power-concavity consequence, obtained by the numerical implication in Lemma \ref{lemma-power-numerical}.

\begin{corollary}\label{cor-power}
    Suppose $d\geq2$. We have
    \begin{equation}\label{eq-power}
        2\alpha_i^{1/(d-1)}\geq \alpha_{i-1}^{1/(d-1)} + \alpha_{i+1}^{1/(d-1)},\quad 1\leq i\leq r-1.
    \end{equation}
    More generally, $(\alpha_i^p)_{i=0}^r$ is concave for every $0<p\leq1/(d-1)$. For $d\geq3$, equality in \eqref{eq-power} holds exactly at a locally constant coefficient triple.
\end{corollary}

Power-concavity is stronger than ordinary log-concavity, while the quadratic inequality contains additional information. For $d=2$, the two displayed inequalities are equivalent at positive internal triples. For $d\geq3$, the quadratic inequality is strictly stronger as an inequality on positive triples.

The form of \eqref{eq-quadratic} is related to the inequality of Berget, Spink, and Tseng for consecutive coefficients of $T_M(1,t)$ \cite[Theorem 11.1]{BST23}.
Note that \eqref{eq-quadratic} is equivalent to
\begin{equation}\label{eq-bst-form}
    (d-1)(\alpha_i^2-\alpha_{i-1}\alpha_{i+1}) + (\alpha_i-\alpha_{i-1})(\alpha_i-\alpha_{i+1})\geq 0.
\end{equation}
Their theorem concerns the Tutte specialization, whereas our coefficients use external semi-activity and realization-dependent determinant weights. Even when every basis weight is one, the two enumerators need not agree. The similarity of the inequalities is explained by their common use of consecutive hypersimplex classes rather than an identification of the polynomials.

\subsection{The mixed-volume representation}

Assume $d\geq2$, so $n=d-1\geq1$. Choose $v\in H$ and set $b_j=a_j-v\in W$.
Give $W$ the volume density determined by
\begin{equation}
    \label{eq-density-intro}
    \omega_W(w_1,\ldots,w_n)\coloneqq\omega_V(w_1,\ldots,w_n,v),
\end{equation}
which by definition does not depend on the choice of $v$.
For $1\leq k\leq N-1$, define the projected hypersimplex
\begin{equation}
    \label{eq-Pk-intro}
    P_k\coloneqq \op{conv}\left\{\sum_{j\in I}b_j\mid I\subset[N],\ |I|=k\right\}\subset W.
\end{equation}
We write $\MV_W$ for mixed volume in the $n$-dimensional space $W$, normalized by $\MV_W(K,\ldots,K)=\vol_W(K)$.
Its definition and the analogous notation for other density spaces are given in Section \ref{section-mixed}.

The key structural result used throughout the paper is Theorem \ref{thm-structural}, which relates mixed volumes of consecutive projected hypersimplices to coefficients of $f_A$.
The following specialization, where every entry occurs once, identifies the individual coefficients. Two specializations with repeated entries then lead to Theorem \ref{thm-main} and Corollary \ref{cor-power}.

\begin{theorem}\label{thm-representation}
    For a flat arrangement of rank $d\geq2$, with $n=d-1$ and $\omega_W$ given by \eqref{eq-density-intro}, we have
    \begin{equation}\label{eq-window-intro}
        \alpha_i=\MV_W(P_{i+1},P_{i+2},\ldots,P_{i+n})\qquad(0\leq i\leq r).
    \end{equation}
\end{theorem}

Galashin, Postnikov, and Williams identified determinant-weighted semi-activity levels with level volume invariants of fine zonotopal tilings \cite[Proposition 4.1, Lemma 5.13, and Proposition 6.2]{GPW22}.
Their slice-volume formula is recovered by the single-body case of our structural theorem.
The consecutive mixed-volume representation supplies the additional geometric information needed for the coefficient inequalities.

\subsection{Further mixed-volume consequences}

The same structural formula controls all placements of a fixed consecutive multiplicity pattern.
For a decomposition $c=(c_1,\dots, c_\ell)$ of $n$, with each $c_j\geq1$, let
\begin{equation}\label{eq-Mi-intro}
    M_i(c)\coloneqq\MV_W(P_{i+1}[c_1],\ldots,P_{i+\ell}[c_\ell]),\quad S_c(t)\coloneqq \sum_{i=0}^{N-\ell-1}M_i(c) t^i,
\end{equation}
where $K[q]$ denotes $q$ separate entries equal to $K$.
The universal mixed-Eulerian polynomial $E_c$ is characterized by
\begin{equation}\label{eq-Ec-intro}
    \frac{E_c(t)}{(1-t)^{n+1}} = \sum_{m\geq 0}\prod_{j=1}^{\ell}(m+j)^{c_j} t^m.
\end{equation}
This identity is due to Berget--Spink--Tseng \cite[Theorem 7.13 and Remark 7.14]{BST23} and Nadeau--Tewari \cite[Proposition 4.5]{NT23}. The resulting factorization and coefficient shape statement is as follows.

\begin{theorem}
    \label{thm-sequences-intro}
    For every decomposition $c$ as above,
    \begin{equation}\label{eq-factor-intro}
        n!S_c(t)= E_c(t) f_A(t),\qquad \sum_{i=0}^{N-\ell-1} M_i(c)=f_A(1).
    \end{equation}
    The polynomial $E_c$ has degree $n-\ell$, positive coefficients, and only negative real zeros unless it is constant. The sequence $(M_i(c))_{i=0}^{N-\ell-1}$ is positive and log-concave.
\end{theorem}

The single-body case $c=(n)$ is the Eulerian convolution for zonotope slice volumes obtained from \cite[Lemma 3.11, proof of Proposition 4.1]{GPW22}.
The formula for arbitrary consecutive multiplicities extends that relation to mixed volumes, which is also a convex-geometric counterpart of the matroid intersection factorization in \cite[Theorem 10.2]{BST23}, replacing the different polynomial $T_M(1,t)$ by $f_A$. The universal factor and its properties are known.
In Section \ref{section-sequences} we recall these notions and give a direct proof.

\subsection{Eulerian digraphs and circulation weights}

A directed multigraph is \emph{connected} when its underlying undirected multigraph is connected. It is \emph{Eulerian} if it is connected and the in-degree and out-degree agree at every vertex.
Given a root $u_0$ and a spanning tree $T$ of the underlying labeled multigraph, let $\kappa_{u_0}(T)$ count the arcs that must be reversed to orient the tree toward $u_0$.
For a connected Eulerian digraph $D$, the Murasugi--Stoimenow polynomial $P_D(t)=\sum_T t^{\kappa_{u_0}(T)}$ is independent of the root \cite[Proposition 1]{MS03}.
Corollary \ref{cor-Eulerian} proves that its coefficients are positive and log-concave, which solves the conjecture of Hafner, M\'esz\'aros, and Vidinas \cite[Conjecture 1.6]{HMV25}.
It also gives the quadratic and power-concavity refinements with rank parameter $\beta(D)=|E(D)|-|V(D)|+1$, where loops are deleted before this parameter is computed.

These conclusions extend to digraphs carrying positive real circulation weights, whose unweighted in-degrees and out-degrees need not agree.
We say that positive arc weights $w_e$ form a \emph{circulation} if
\begin{equation}\label{eq-balance-intro}
    \sum_{\op{tail}(e)=u} w_e = \sum_{\op{head}(e)=u} w_e\quad \text{ at every vertex } u.
\end{equation}
The corresponding weighted tree polynomial is
\begin{equation}\label{eq-PDw-intro}
    P_{D,w}(t)=\sum_T \left(\prod_{e\in T} w_e\right) t^{\kappa_{u_0}(T)} = \sum_{i=0}^{\nu-1} c_i(D,w) t^i,\quad \nu=|V(D)|.
\end{equation}
The loops may be deleted, and arcs with the same ordered endpoints may be replaced by a single arc carrying their total weight.
Let $m_{\red}$ be the number of remaining arcs and $\beta_{\red}=m_{\red}-\nu+1$.
Anti-parallel arcs remain distinct.

\begin{theorem}\label{thm-weighted-intro}
    For a connected directed multigraph with a positive real circulation, $P_{D,w}$ is independent of the root and has positive, palindromic coefficients. If $\beta_{\red}\geq1$, then at every internal index
    \begin{equation}\label{eq-weighted-intro}
        \beta_{\red} c_i^2\geq c_i(c_{i-1}+c_{i+1})+(\beta_{\red}-2)c_{i-1}c_{i+1}.
    \end{equation}
    The coefficients are log-concave, with equality only at a locally constant triple. For $\beta_{\red}\geq 2$, their $1/(\beta_{\red}-1)$-st powers form a concave sequence; for $\beta_{\red}=1$, the coefficients are constant.
    The trivial case of one vertex without edges has polynomial $1$.
\end{theorem}

Unit weights on an Eulerian digraph satisfy \eqref{eq-balance-intro}, so Theorem \ref{thm-weighted-intro} includes the unweighted result.
Example \ref{example-nonEulerian-support} exhibits a positive circulation on a support that is not Eulerian.
In Section \ref{section-graphs} we prove the weighted polynomial identity by rescaling the columns of a cographic matrix. The construction applies directly to positive real weights.

\subsection{Special alternating links}

Consider a connected reduced special alternating diagram of a link $L$, with $c$ crossings and $s$ Seifert circles. The diagram conventions are recalled in Section \ref{section-links}.
We write the positive support normalization of its Alexander polynomial as
\begin{equation}
    \label{eq-link-intro-poly}
    F_L(t)\eqdot\Delta_L(-t),\qquad F_L(t)=\sum_{i=0}^{c-s+1}\gamma_i t^i,
\end{equation}
where $\eqdot$ means equality up to a unit $\pm t^k$.

\begin{theorem}\label{thm-links-intro}
    The coefficients $\gamma_i$ are positive and palindromic, and
    \begin{equation}
        \label{eq-link-intro}
        (s-1)\gamma_i^2\geq\gamma_i(\gamma_{i-1}+\gamma_{i+1}) + (s-3)\gamma_{i-1}\gamma_{i+1},\quad 1\leq i\leq c-s.
    \end{equation}
    For $s\geq3$, the sequence $(\gamma_i^{1/(s-2)})$ is concave.
    The sequence $(\gamma_i)$ is log-concave, with equality exactly at a locally constant triple.
    The same equality characterization holds for power-concavity when $s\geq4$.
\end{theorem}

Ordinary log-concavity for special alternating links was already proved in \cite{HMV24}.
Theorem \ref{thm-links-intro} refines it by the diagram-dependent quadratic inequality, a positive power-concavity exponent, and equality rigidity.
Its parameter comes from the cycle rank of the dual Eulerian digraph, which equals $s-1$.

\subsection{Organization}

Section \ref{section-mixed} fixes the mixed-volume conventions and evaluates three hypersimplex constants.
Section \ref{section-projected} records the geometry of the projected hypersimplices.
Section \ref{section-upper} proves the upper-cell formula, classifies the contributing cells, and establishes Theorem \ref{thm-structural}.
Section \ref{section-inequalities} proves Theorem \ref{thm-main} and Corollary \ref{cor-power}, and derives quantitative growth bounds from the ratio inequality.
Section \ref{section-sequences} treats consecutive mixed-volume sequences.
Sections \ref{section-graphs} and \ref{section-links} give the graph and special-alternating link applications.

\vskip0.5cm
\noindent \textit{Acknowledgements}.
The authors thank Yin Tian for introducing them to this project and for helpful discussions. They acknowledge the use of ChatGPT in exploratory mathematical discussions, for assistance with developing and checking arguments. The authors subsequently revised the exposition and take full responsibility for the mathematical content and the final text.

\section{Mixed volumes and hypersimplices}
\label{section-mixed}

We first recall the definition of mixed-volume in an arbitrary density space, and then apply them to the universal hypersimplices. 
For the general statements in this section only, we use a separate notation $U$ to denote a vector space of dimension $m\geq1$, which later will be either $W$, a root-lattice space $E_n$, or a lower-dimensional space arising from a face.

\subsection{Mixed-volume conventions}
\label{subsection-MV-conventions}

Equip $U$ with the density $|\omega_U|$ of a nonzero alternating $m$-form $\omega_U$ and let $\op{vol}_U$ denote its volume.
A convex body means a nonempty compact convex set, possibly with empty interior.
We use $K_j$ for convex bodies, $Q_j$ for polytopes in the facet formula, and $\mu_n$ for the scalar mixed volume used in the induction below.
For convex bodies $K_1,\dots,K_q\subset U$ and nonnegative $t_j$, their Minkowski combination is
\begin{equation*}
    \sum_{j=1}^q t_j K_j=\left\{\sum_{j=1}^q t_j x_j \mid x_j\in K_j\right\}.
\end{equation*}
The mixed volume $\MV_U(K_1,\dots,K_m)\in\R$ is normalized by $\MV_U(K,\dots,K)=\vol_U(K)$ and characterized by
\begin{equation}\label{eq-polarization}
    \op{vol}_U\left(\sum_{j=1}^q t_j K_j\right)=\sum_{j_1,\dots,j_m=1}^q t_{j_1} \cdots t_{j_m}\MV_U(K_{j_1},\dots,K_{j_m}).
\end{equation}
For $m$ separately indexed summands, the coefficient of $t_1\cdots t_m$ is $m!\MV_U(K_1,\dots,K_m)$.
This remains true when some summands coincide. Mixed volume is nonnegative, symmetric, translation invariant, monotone in each entry, and multilinear for Minkowski addition and nonnegative scaling. See \cite[Chapter 5]{Schneider14}.
We write $K[a]$ for $a$ consecutive copies of $K$ and omit this entry when $a=0$.
The mixed volume is zero if some entry is a single point. To see this, we translate the point to the origin, and the volume polynomial becomes independent of the corresponding parameter.

For $m\geq2$, the Alexandrov--Fenchel inequality takes the form
\begin{equation}\label{eq-AF}
    \MV_U(K_1, K_2, K_3,\dots, K_m)^2\geq\MV_U(K_1, K_1, K_3,\dots,K_m)\MV_U(K_2, K_2, K_3,\dots,K_m).
\end{equation}
The common list $K_3,\dots,K_m$ is empty when $m=2$.
Note that the bodies in \eqref{eq-AF} may have empty interior \cite[Theorem 1.1 and Section 2.1]{SvH19}.
In dimension one, mixed volume has only one entry and is simply the length.

\subsection{Exposed faces and quotient densities}
\label{subsection-facet}

For a polytope $Q\subset U$ and a functional $u\in U^*$, let
\begin{equation*}
    h_Q(u)=\max_{x\in Q}u(x),\qquad Q^u=\{x\in Q \mid u(x)=h_Q(u)\},
\end{equation*}
i.e., $h_Q(u)$ is a scalar support value and $Q^u$ is an exposed face.
Assume $m\geq2$ and $u\neq0$. Choose $z_u\in U$ with $u(z_u)=1$ and define a form on the $(m-1)$-dimensional space $\ker u$ by
\begin{equation}
    \label{eq-quotient-density}
    \omega_{U,u}(x_1,\dots,x_{m-1})\coloneqq\omega_U(x_1,\dots,x_{m-1},z_u).
\end{equation}

Note that $\omega_{U,u}$ does not depend on the choice of $z_u$: replacing $z_u$ by another such vector changes the last entry by a vector in $\ker u$, so the added alternating-form value is zero. Therefore, the density $|\omega_{U,u}|$ is well defined.
Let $\op{vol}_{U,u}$ and $\MV_{U,u}$ be the volume and mixed volume on $\ker u$ with this density. 
Each face $Q^u$ lies in an affine hyperplane parallel to $\ker u$, and we translate it into $\ker u$ before computing $\MV_{U,u}$.
Translation invariance makes the chosen translations irrelevant.
If $u$ is replaced by $bu$, $b>0$, then the quotient density is divided by $b$, while $h_Q(bu)=b h_Q(u)$, hence their product is unchanged.

For the facet formula, let $Q_1,\dots,Q_m$ be full-dimensional polytopes in $U$.
Choose a complete polyhedral fan refining their normal fans and one nonzero functional $u$ on each of its rays.
The polytope facet formula, with the quotient density just defined, is
\begin{equation}
    \label{eq-facet}
    m\MV_U(Q_1,\dots,Q_m)=\sum_u h_{Q_m}(u) \MV_{U,u}(Q_1^u,\dots,Q_{m-1}^u).
\end{equation}
Note that the additional rays introduced by refinement are harmless: if the sum of the corresponding faces has dimension less than $m-1$, the mixed-volume contribution is zero.

We recall the polarization argument to fix the factor $m$ in \eqref{eq-facet}.
For $t_1,\dots,t_{m-1}>0$, denote $Q(t)=\sum_{j=1}^{m-1}t_j Q_j$.
The polytopal first-variation formula for volume is
\begin{equation}\label{eq-first-variation}
    \left.\frac{d}{ds}\right|_{s=0+} \op{vol}_U \left(Q(t)+ s Q_m\right)=\sum_u h_{Q_m}(u)\op{vol}_{U,u} \left(Q(t)^u\right).
\end{equation}
See \cite[Section 5.1]{Schneider14} for details.
With the volume form \eqref{eq-quotient-density}, the product on the right-hand side is the support displacement times the corresponding facet volume, so no Euclidean unit normal factor is missed.
Since $Q(t)^u=\sum_{j=1}^{m-1} t_j Q_j^u$, both sides of \eqref{eq-first-variation} are polynomials in $t$.
The coefficient of $t_1\cdots t_{m-1}$ on the left-hand side is $m!\MV_U(Q_1,\dots,Q_m)$, whereas that coefficient on the right-hand side is $(m-1)!$ times the sum in \eqref{eq-facet}.
Dividing by $(m-1)!$ then proves \eqref{eq-facet}.
We will not use this formula in dimension one, so no zero-dimensional mixed-volume convention is needed.

\subsection{Three hypersimplex constants}
\label{subsection-hypersimplex-constants}

For each $n\geq1$, let
\begin{equation*}
    E_n=\left\{x\in\R^{n+1} \mid \sum_{j=1}^{n+1}x_j=0\right\},\qquad \Lambda_n = E_n\cap\Z^{n+1},
\end{equation*}
and give $E_n$ the volume density for which the root lattice $\Lambda_n$ has covolume one.
These are universal $n$-dimensional density spaces, separate from $W$, In applications both spaces have dimension $n=d-1$.
Define the centered hypersimplices
\begin{equation}\label{eq-hypersimplices}
    D_k^{(n)}=\left\{x\in[0,1]^{n+1} \mid \sum_jx_j=k\right\}-\frac{k}{n+1}\one,\qquad 1\leq k\leq n,
\end{equation}
which are of full dimension in $E_n$.

\begin{lemma}
    \label{lemma-constants}
    For $n\geq1$,
    \begin{equation}\label{eq-constant-main}
        \MV_{E_n}(D_1^{(n)}, D_2^{(n)},\dots,D_n^{(n)})=1.
    \end{equation}
    For $n\geq2$,
    \begin{gather}
        \MV_{E_n}(D_1^{(n)}[2], D_2^{(n)},\dots, D_{n-1}^{(n)})=\frac{1}{n}, \label{eq-constant-left}\\
        \MV_{E_n}(D_2^{(n)}[2], D_3^{(n)},\dots, D_n^{(n)})=\frac{n-1}{n}. \label{eq-constant-right}
    \end{gather}
\end{lemma}
\begin{proof}
    First we describe the exposed faces and their volume density.
    We then compute the three mixed volumes.
    The braid fan whose maximal cones correspond to the order of the coordinates refines the normal fans of all the hypersimplices.
    Its rays are represented by
    \begin{equation*}
        u_S(x) = \sum_{j\in S} x_j, \qquad \varnothing \neq S \subsetneq [n+1].
    \end{equation*}
    
    Note that
    \begin{equation}\label{eq-hypersimplex-support}
        h_{D_k^{(n)}}(u_S)= \op{min}(k,s)-\frac{ks}{n+1},
    \end{equation}
    where $s=|S|$. Indeed, to maximize $u_S$ in the uncentered hypersimplex, we place as much of the total coordinate sum $k$ as possible in $S$.
    If $k<s$, all coordinates outside $S$ are zero and the coordinates in $S$ form a hypersimplex of coordinate sum $k$.
    If $k=s$, the maximizing point is $\one_S$.
    If $k>s$, every coordinate in $S$ is one and the complementary coordinates form a hypersimplex of coordinate sum $k-s$.
    Centering translates these faces such that
    \begin{equation}\label{eq-point-face}
        (D_s^{(n)})^{u_S}=\{\one_S-\frac{s}{n+1}\one\}.
    \end{equation}

    The quotient density on $\ker u_S\subset E_n$ gives covolume one to $\Lambda_n\cap\ker u_S$.
    To check this, choose $p\in S$ and $q\notin S$.
    Then $u_S(e_p-e_q)=1$, so adjoining $e_p-e_q$ to a lattice basis of $\Lambda_n\cap \ker u_S$ gives a lattice basis of $\Lambda_n$.
    This is exactly the normalization in \eqref{eq-quotient-density}.
    For $s=n$, the kernel is the zero sum coordinate space on $S$. 
    For $s=1$, it is the zero sum coordinate space on $S^c$.
    In both cases a coordinate relabeling identifies it with $E_{n-1}$ and preserves the root-lattice density.
    Thus the lower-dimensional face mixed volumes in the two surviving cases below are computed using precisely $\MV_{E_{n-1}}$.

    Denote
    \begin{equation*}
        \mu_n = \MV_{E_n}(D_1^{(n)},\dots,D_n^{(n)}).
    \end{equation*}
    The segment $D_1^{(1)}$ has endpoints whose difference is the primitive lattice vector $e_1-e_2$, hence $\mu_1=1$.
    For $n\geq2$, apply \eqref{eq-facet} in $U=E_n$ with the first $n-1$ entries $D_1^{(n)},\dots,D_{n-1}^{(n)}$ and the last entry $D_n^{(n)}$.
    For $1\leq s\leq n-1$, the face list contains the point \eqref{eq-point-face}, so its $(n-1)$-dimensional mixed volume vanishes.
    For $s=n$, the first $n-1$ faces are translations of $D_1^{(n-1)},\dots,D_{n-1}^{(n-1)}$ in the zero sum coordinates on $S$.
    Their mixed volume is $\mu_{n-1}$ by the density computation above.
    There are $n+1$ subsets of size $n$, and each has support value $n/(n+1)$ for the last entry.
    Therefore,
    \begin{equation*}
        n\mu_n=(n+1)\frac{n}{n+1}\mu_{n-1}=n\mu_{n-1},
    \end{equation*}
    which implies $\mu_n=\mu_{n-1}=\dots=\mu_1=1$ and proves \eqref{eq-constant-main}.

    For \eqref{eq-constant-left}, keep the same first $n-1$ entries and replace the last entry by $D_1^{(n)}$.
    The same point face argument leaves only $s=n$.
    The quotient face mixed volume remains $\mu_{n-1}=1$ and the last support value is now $1/(n+1)$.
    Therefore,
    \begin{equation*}
        n\MV_{E_n}(D_1^{(n)}[2], D_2^{(n)},\dots,D_{n-1}^{(n)})=(n+1)\frac{1}{n+1}\mu_{n-1}=1.    
    \end{equation*}
    For \eqref{eq-constant-right}, use $D_2^{(n)},\dots,D_n^{(n)}$ as the first $n-1$ entries and $D_2^{(n)}$ as the last entry.
    Each term with $2\leq s\leq n$ contains the point face of $D_s^{(n)}$ and vanishes.
    For $s=1$, the first $n-1$ faces are translations of $D_1^{(n-1)},\dots,D_{n-1}^{(n-1)}$ on the complementary coordinates.
    Their quotient mixed volume is again $\mu_{n-1}=1$.
    There are $n+1$ singleton subsets, and the support value of the last entry is $(n-1)/(n+1)$.
    Consequently,
    \begin{equation*}
        n\MV_{E_n}(D_2^{(n)}[2],D_3^{(n)},\dots,D_n^{(n)})=(n+1)\frac{n-1}{n+1}\mu_{n-1}=n-1.
    \end{equation*}
\end{proof}

The map $x\mapsto -x$ preserves the volume density on $E_n$ and sends $D_k^{(n)}$ to $D_{n+1-k}^{(n)}$.
We will also use the reflected forms of \eqref{eq-constant-left}--\eqref{eq-constant-right}.
These values agree with the corresponding mixed Eulerian numbers in \cite[Definition 16.1 and Theorem 16.3]{Postnikov09}.

\section{Projected hypersimplices}\label{section-projected}

Consider the flat arrangement \eqref{eq-arrangement} of Section \ref{section-introduction} and assume $d\geq2$, where $V$ has dimension $d$, $W=\ker\varphi$ has dimension $n=d-1$, and $N=r+n+1$.
Choose $v\in H$ and write $b_j=a_j-v\in W$.
The form $\omega_V$ measures the determinant weights in $V$, while the form $\omega_W$ in \eqref{eq-density-intro} measures volumes in $W$.

\begin{lemma}\label{lemma-density}
    The form $\omega_W$ in \eqref{eq-density-intro} is independent of $v\in H$.
\end{lemma}
\begin{proof}
    Replacing $v$ by $v+w$, $w\in W$, adds a term $\omega_V(w_1,\dots,w_n,w)$, which vanishes because all $n+1=d$ entries lie in the $n$-dimensional space $W$.
\end{proof}

All mixed volumes of projected hypersimplices are denoted by $\MV_W$ and computed with the density $|\omega_W|$.
For convenience, we also define the two endpoint polytopes:
\begin{equation}\label{eq-Pk}
    P_k=\left\{\sum_{j=1}^N x_jb_j\mid 0\leq x_j\leq1,\, \sum_jx_j=k\right\},\qquad 0\leq k\leq N.
\end{equation}
For integer $k$, the vertices of the hypersimplex in the source are the incidence vectors of $k$-subsets, so \eqref{eq-Pk} agrees with \eqref{eq-Pk-intro}.
In particular, $P_0=\{0\}$ and $P_N=\{b_{[N]}\}$, where $b_I=\sum_{j\in I} b_j$.
The flatness gives an equivalent slice description
\begin{equation}\label{eq-zonotope-slice}
    P_k=(A([0,1]^N)\cap\varphi^{-1}(k))-kv.
\end{equation}
Indeed, $\varphi(Ax)=\sum_j x_j$ for every $x$.
Changing $v$ translates each $P_k$ and therefore leaves all its mixed volumes unchanged.

\begin{lemma}\label{lemma-projected-properties}
    For $1\leq k\leq N-1$, the polytope $P_k$ has nonempty interior in $W$.
    Moreover,
    \begin{equation}\label{eq-reflection}
        P_{N-k}=b_{[N]}-P_k, \qquad 0\leq k\leq N,
    \end{equation}
    and
    \begin{equation}\label{eq-midpoint-inclusion}
        \frac{1}{2}(P_{k-1}+P_{k+1})\subseteq P_k, \qquad 1\leq k\leq N-1.
    \end{equation}
\end{lemma}
\begin{proof}
    The restriction of $A$ to the zero-sum subspace of $\R^N$ maps onto $W$: if $Ax=w\in W$, then $\sum_j x_j=\varphi(Ax)=0$.
    The point $(k/N)\one$ lies in the relative interior of the source hypersimplex for $1\leq k\leq N-1$.
    A surjective linear map sends a relative neighborhood of this point to a neighborhood in $W$, proving full dimensionality.
    Complementing the source coordinates, $x\mapsto\one-x$, proves \eqref{eq-reflection}.
    Averaging a source point of coordinate sum $k-1$ and one of sum $k+1$ gives a point in $[0,1]^N$ with coordinate sum $k$, which proves \eqref{eq-midpoint-inclusion}.
\end{proof}

For a weak composition $q=(q_1,\dots,q_n)$ of $n$, define
\begin{equation}\label{eq-ME}
    \ME(q)=n!\MV_{E_n}(D_1^{(n)}[q_1],\dots,D_n^{(n)}[q_n]).
\end{equation}
An entry $q_k=0$ means that $D_k^{(n)}$ is omitted.
Note that this mixed volume is computed in the root-lattice density space $E_n$ instead of $W$.
It has $n$ entries, counted with multiplicity, and follows the normalization of \cite[Definition 16.1]{Postnikov09}. The numbers $\ME(q)$ are positive integers.

\section{The upper-envelope formula}\label{section-upper}

This section proves the structural theorem, Theorem \ref{thm-structural}.
We first obtain a parameter-independent subdivision directly from lifted Minkowski sums, then identify its nonzero mixed-volume contributions for consecutive projected hypersimplices.
Throughout this section, $W$ has dimension $n=d-1\geq1$.
The lifts lie in $W\oplus\R$, which is of dimension $n+1=d$, but all volumes in the upper-cell formula are computed after projection to $W$.

\subsection{Additivity over upper cells}

For a polytope $\widehat{K}\subset W\oplus\R$, we write $K=\op{pr}_W(\widehat{K})$. The $\R$-coordinate denotes height.
An \emph{upper face} is exposed by a functional whose coefficient on height is positive. After normalization, it has the form $(u,1)$ with $u\in W^*$.
We allow the entire polytope to be an exposed face when the exposing functional is constant on it.
We use the notation
\begin{equation*}
    \widehat{F}_K(u)=\widehat{K}^{(u,1)},\qquad F_K(u)=\op{pr}_W (\widehat{F}_K(u)).
\end{equation*}
Here $F_K(u)$ is a projected lifted face and need not be an exposed face of $K$.
This notation is distinct from the ordinary exposed face notation $Q^u$ in Section \ref{subsection-facet}.
For a nonempty polytope $F$, we define $\op{dir} F=\op{span}(F-F)$.

\begin{lemma}\label{lemma-upper-cells}
    Let $\widehat{K}_1,\dots,\widehat{K}_n\subset W\oplus\R$ be nonempty polytopes such that $K_1+\dots+K_n$ has nonempty interior.
    There is a finite set $\mathcal{C}$ of tuples $F=(F_1,\dots,F_n)$ of projected faces selected by a common upper normal, independent of $\lambda\in\R_{>0}^n$, for which
    \begin{equation*}
        C_F(\lambda)=\lambda_1 F_1+\dots+\lambda_n F_n\qquad (F\in\mathcal{C})
    \end{equation*}
    are exactly the distinct full-dimensional projected upper faces of $\sum_j\lambda_j\widehat{K}_j$, which cover $\sum_j\lambda_j K_j$ and have pairwise disjoint interiors.
    Moreover,
    \begin{equation}\label{eq-upper-additivity}
        \MV_W(K_1,\dots,K_n)=\sum_{F\in\mathcal{C}} \MV_W(F_1,\dots,F_n).
    \end{equation}
    We treat coincident lifted polytopes as separately indexed entries.
\end{lemma}
\begin{proof}
    For positive $\lambda_j$, let $\widehat{K}(\lambda)=\sum_j\lambda_j \widehat{K}_j$ and $K(\lambda)=\sum_j \lambda_j K_j$.
    The exposed faces of Minkowski sums satisfy
    \begin{equation}\label{eq-face-additivity}
        \widehat{K}(\lambda)^{(u,1)}=\sum_j\lambda_j\widehat{K}_j^{(u,1)}.
    \end{equation}
    Therefore, the tuple $F_j=F_{K_j}(u)$ selected by $u$ is independent of $\lambda$.
    Its projected dimension is independent of $\lambda$ as well since
    \begin{equation}\label{eq-directions}
        \op{dir} C_F(\lambda)=\sum_j\op{dir} F_j,\qquad \lambda_j>0.
    \end{equation}
    To verify both inclusions, note that $0\in F_j-F_j$ for every $j$ and that
    \begin{equation*}
        C_F(\lambda)-C_F(\lambda)=\sum_j\lambda_j(F_j-F_j).
    \end{equation*}
    For each $i$, this gives
    \begin{equation*}
        \lambda_i(F_i-F_i)\subseteq C_F(\lambda)-C_F(\lambda)\subseteq\sum_j\op{dir} F_j.
    \end{equation*}
    Taking linear spans in the second inclusion proves $\op{dir} C_F(\lambda)\subseteq\sum_j\op{dir} F_j$.
    Taking spans in the first and using $\lambda_i>0$ gives $\op{dir} F_i\subseteq\op{dir} C_F(\lambda)$ for every $i$, which proves the reverse inclusion.
    The same identity applied to the $K_j$ shows that $K(\lambda)$ is full-dimensional for every positive $\lambda$.
    There are only finitely many face tuples, so let $\mathcal{C}$ consist of those selected by an upper normal for which $\sum_j\op{dir} F_j=W$.
    Equations \eqref{eq-face-additivity}--\eqref{eq-directions} show that this set indexes the full-dimensional projected upper faces for every positive $\lambda$.

    We next show that these projected faces cover $K(\lambda)$ and have disjoint interiors, so that we can apply volume additivity.
    For this purpose, we define the upper roof
    \begin{equation*}
        g_\lambda(x)=\max\{z\mid (x,z)\in\widehat{K}(\lambda)\},\qquad x\in K(\lambda).
    \end{equation*}
    The maximum exists because every fiber is nonempty and compact.
    For any $u\in W^*$, the supporting inequality for $\widehat{K}(\lambda)$ gives
    \begin{equation*}
        g_\lambda(x)\leq h_{\widehat{K}(\lambda)}(u,1)-u(x).
    \end{equation*}
    Equality holds exactly when the roof point $(x,g_\lambda(x))$ lies in the exposed upper face.
    Consequently,
    \begin{equation}\label{eq-roof-contact}
        \op{pr}_W(\widehat{K}(\lambda)^{(u,1)}) = \{x\in K(\lambda)\mid g_\lambda(x)=h_{\widehat{K}(\lambda)}(u,1)-u(x)\}.
    \end{equation}
    Therefore, the projected upper faces are exactly the contact sets of the roof with these affine supporting functions.

    Since $K(\lambda)$ has dimension $n$, the lifted polytope $\widehat{K}(\lambda)$ has dimension either $n$ or $n+1$.
    Suppose first that its dimension is $n+1$.
    Write the inequalities of its upper facets as $u_\nu(x) + z\leq\eta_\nu$, with $\nu$ in a finite index set $J$.
    This set is nonempty because the polytope is bounded above.
    For $x\in K(\lambda)$, start with any feasible height $z_0$.
    Raising $z_0$ to $\min_{\nu\in J}(\eta_\nu-u_\nu(x))$ preserves the inequalities whose coefficient on height is nonpositive, and satisfies all the upper-facet inequalities.
    Hence
    \begin{equation}\label{eq-roof-minimum}
        g_\lambda(x)=\min_{\nu\in J}\{\eta_\nu-u_\nu(x)\},\qquad x\in K(\lambda).
    \end{equation}
    In particular, $g_\lambda$ is concave and piecewise affine.
    Each contact set where one of the affine functions in \eqref{eq-roof-minimum} attains the minimum is the projection of the corresponding upper facet by \eqref{eq-roof-contact}.
    Projection is an affine isomorphism on the hyperplane $u_\nu(x)+z=\eta_\nu$, so each such contact set has dimension $n$. They cover $K(\lambda)$, including its boundary, since a finite minimum is attained at every point.
    If two of their interiors overlap, their affine functions will agree on an open set and hence everywhere, which makes the two facets equal. Their interiors are therefore disjoint.
    These are all the full-dimensional projected upper faces: projection is injective on any upper face, so such a face has dimension $n$ and is a facet.

    If instead $\dim\widehat{K}(\lambda)=n$, then the projection from its affine hull onto $W$ is an affine isomorphism. The lifted polytope is therefore the graph of the affine function $g_\lambda$ on $K(\lambda)$. Its graph is exposed by the upper normal determined by its affine slope.
    The only full-dimensional projected upper face is then $K(\lambda)$ itself.
    These are the only two possible dimensions, so the projected faces cover the polytope and have disjoint interiors.

    We also use the roof to show that distinct tuples give distinct cells.
    If two tuples gave the same full-dimensional projected cell, \eqref{eq-roof-contact} would express $g_\lambda$ there by two affine functions.
    Equality on an open set forces their linear parts and hence their normalized upper normals to agree.
    Each summand has a unique exposed face at that normal, so the two tuples coincide.

    Volume additivity now gives, for every $\lambda\in\R_{>0}^n$,
    \begin{equation}\label{eq-volume-polynomial-additivity}
        \op{vol}_W\left(\sum_j\lambda_j K_j\right)=\sum_{F\in\mathcal{C}}\op{vol}_W\left(\sum_j\lambda_j F_j\right).
    \end{equation}
    Both sides are homogeneous polynomials of degree $n$, and the indexing set $\mathcal{C}$ is fixed throughout the positive orthant. Their coefficients therefore agree.
    By \eqref{eq-polarization} with $U=W$ and $m=n$, the coefficient of $\lambda_1\cdots\lambda_n$ is $n!$ times the corresponding $\MV_W$, which proves \eqref{eq-upper-additivity}.
    Lower-dimensional projected faces have zero $n$-dimensional volume.
    Separate parameters are used even when two summands coincide, so the argument introduces no additional multiplicity factors.
\end{proof}

\subsection{Scores and basis faces}

Fix the circuit-generic vector $\rho$ used in \eqref{eq-activity}.
For $1\leq k\leq N-1$, lift $P_k$ to
\begin{equation}\label{eq-lifted-P}
    \widehat{P}_k=\op{conv}\{(b_I,\rho_I)\mid I\subset[N],\, |I|=k\},\qquad  \rho_I=\sum_{j\in I}\rho_j.
\end{equation}
We assign the height to each labeled subset rather than to its projected image.
In particular, two subsets with the same $b_I$ may have different heights.
For $u\in W^*$, recall that the \emph{score} of label $j$ at $u$ is
\begin{equation}
    \label{eq-scores}
    s_j(u)=\rho_j + u(b_j).
\end{equation}
We use the word ``score'' for this quantity throughout the paper.
Let $F_k(u)=\op{pr}_W(\widehat{P}_k^{(u,1)})$ be the corresponding projected upper face.
For a labeled $k$-subset $I$,
\begin{equation*}
    (u,1)(b_I, \rho_I)=u(b_I)+\rho_I=\sum_{j\in I}(\rho_j+u(b_j))=\sum_{j\in I}s_j(u).
\end{equation*}
Therefore, a maximizing subset consists of $k$ labels with largest scores. If several labels are tied at the last selected value, we may choose any of them.

For fixed $u$, we call an equality class of the values $s_1(u),\dots,s_N(u)$ a \emph{tied score block}.
Order these blocks from highest to lowest score.
For a block $B$, let $T$ be the labels with higher score and let $p=|T|$.
For $p<k<p+|B|$,
\begin{equation}
    \label{eq-score-face}
    F_k(u)=b_T+ \left\{\sum_{j\in B} x_j b_j \mid 0\leq x_j\leq 1,\, \sum_{j\in B}x_j=k-p \right\}.
\end{equation}
At a boundary between blocks, the maximizing subset is unique, so $F_k(u)$ is a point.
We call the integer $k$ the selected cut in the ordered list of scores.

\begin{lemma}\label{lemma-tied-block}
    The columns indexed by any tied score block are linearly independent.
    In particular, a tied block has at most $d=n+1$ labels.
\end{lemma}
\begin{proof}
    Suppose a block is dependent. Then it supports a circuit vector $z$.
    Flatness gives $\sum_j z_j=\varphi(Az)=0$ and $\sum_j z_j b_j=0$.
    If the common score on the block is $q$, then
    \begin{equation*}
        \langle\rho, z\rangle=\sum_j z_j(q-u(b_j))=0,
    \end{equation*}
    which contradicts circuit genericity.
\end{proof}

\begin{lemma}\label{lemma-basis-tie}
    For every basis $B$, there is a unique $(u_B, q_B)\in W^*\times\R$ such that
    \begin{equation}\label{eq-basis-tie}
        \rho_j+u_B(b_j)=q_B, \qquad j\in B.
    \end{equation}
    The block of score $q_B$ consists exactly of $B$, and the labels above it are $\Ext_\rho(B)$.
    The linear map
    \begin{equation}\label{eq-TB}
        T_B\colon E_B=\left\{x\in \R^B \mid \sum_{j\in B} x_j=0\right\}\to W,\qquad T_B(x)=\sum_{j\in B} x_j b_j,
    \end{equation}
    which uses the root-lattice density on $E_B$, is an isomorphism with volume factor $w_A(B)$. Moreover, both $E_B$ and $W$ have dimension $n$.
\end{lemma}
\begin{proof}
    The points $b_j$, $j\in B$, form an affine basis of $W$, so \eqref{eq-basis-tie} has at most one solution.
    To construct the solution, we order $\rho_B$ according to the columns of $A_B$ and define
    \begin{equation*}
        \ell_B=\rho_B A_B^{-1}\in V^*,\qquad q_B=\ell_B(v),\qquad u_B=-\ell_B|_W.
    \end{equation*}
    Since $b_j=a_j-v$, for every label $j$ we have
    \begin{equation*}
        \rho_j+u_B(b_j)-q_B=\rho_j-\ell_B(a_j).
    \end{equation*}
    The right-hand side vanishes for $j\in B$ and proves \eqref{eq-basis-tie}.
    For $j\notin B$, it gives
    \begin{equation}\label{eq-activity-sign}
        s_j(u_B)-q_B=\rho_j-\rho_B A_B^{-1}a_j=\langle \rho, c^{B,j}\rangle.
    \end{equation}
    By circuit genericity, this quantity is nonzero. It is positive exactly when $j\in\Ext_\rho(B)$.
    This proves the statements about the tied block and the labels above it.

    Order $B=\{j_1,\dots,j_d\}$ and put
    \begin{equation*}
        \varepsilon_h=e_{j_h}-e_{j_d},\qquad \delta_h=T_B(\varepsilon_h)=a_{j_h}-a_{j_d}, \qquad 1\leq h\leq n.
    \end{equation*}
    The set $\{\varepsilon_h\}$ forms a basis of the root lattice $E_B\cap\Z^B$, and $\{\delta_h\}$ forms a basis of $W$.
    The volume factor of $T_B$ is therefore $|\omega_W(\delta_1,\dots,\delta_n)|$.
    Using the induced density on $W$, we compute
    \begin{align}\label{eq-det-factor}
        |\omega_W(\delta_1,\dots,\delta_n)|
        &=|\omega_V(\delta_1,\dots,\delta_n,v)|\nonumber\\
        &=|\omega_V(\delta_1,\dots,\delta_n,a_{j_d})|\nonumber\\
        &=|\omega_V(a_{j_1},\dots,a_{j_d})|\nonumber\\
        &=w_A(B).
    \end{align}
    The second equality follows from $a_{j_d}-v\in W$. For the third equality, we add the last column to each of the first $n$ columns.
    In particular, polarization of the volume change gives, for convex bodies $L_1,\dots,L_n\subset E_B$,
    \begin{equation}\label{eq-TB-mixed-volume}
        \MV_W(T_B L_1,\dots,T_B L_n)=w_A(B)\MV_{E_B}(L_1,\dots,L_n).
    \end{equation}
\end{proof}

\subsection{Classification of the contributing cells}

The following proposition is the main geometric step and its proof explains why the selected cuts must be consecutive.

\begin{proposition}\label{prop-classification}
    Fix integers $s\geq 1$ and $\ell\geq 1$ with $s+\ell-1\leq N-1$, and positive integers $c_1,\dots,c_\ell$ with sum $n$. Apply Lemma \ref{lemma-upper-cells} to the list containing $c_j$ separately indexed copies of $\widehat{P}_{s+j-1}$. The nonzero summands in \eqref{eq-upper-additivity} are in bijection with the bases $B$ that satisfy
    \begin{equation}\label{eq-contributing-condition}
        0\leq s-\ext_\rho(B)-1\leq n-\ell.
    \end{equation}
    The summand corresponding to $B$ is the face tuple selected by $u_B$ and has value
    \begin{equation*}
        \MV_W(F_s(u_B)[c_1],\dots, F_{s+\ell-1}(u_B)[c_\ell]).
    \end{equation*}
    More precisely, let $T=\Ext_\rho(B)$ and $p=|T|$. After identifying the coordinates of $E_n$ with the labels in $B$, we have
    \begin{equation}\label{eq-local-hypersimplex}
        F_k(u_B)=b_T+\frac{k-p}{d}b_B+T_B(D_{k-p}^{(n)}), \qquad s\leq k\leq s+\ell-1.
    \end{equation}
    Each basis satisfying the stated condition gives exactly one summand. This also holds when some columns or projected subset sums coincide.
\end{proposition}

\begin{proof}
    Consider a nonzero summand in \eqref{eq-upper-additivity}, and let $u$ be the parameter of its common upper normal.
    Its local mixed volume is
    \begin{equation*}
        \MV_W(F_s(u)[c_1],\dots,F_{s+\ell-1}(u)[c_\ell]).
    \end{equation*}
    None of the selected faces can be a point.
    By \eqref{eq-score-face} and the observation about block boundaries, each cut $s,s+1,\dots,s+\ell-1$ lies strictly inside a tied score block.
    All these cuts lie in the same block: otherwise a boundary between two blocks will be an integer between two selected cuts, hence will itself be selected and give a point face.
    We denote the common block by $B$.

    By \eqref{eq-score-face}, all selected face directions lie in
    \begin{equation*}
        \op{span}\{b_j-b_k\mid j, k\in B\}.
    \end{equation*}
    If this space had dimension less than $n$, every positive Minkowski combination of the faces would have zero $n$-dimensional volume, so the local mixed volume would be zero.
    It therefore has dimension $n$.
    Therefore, $|B|\geq n+1$, while Lemma \ref{lemma-tied-block} gives $|B|\leq d=n+1$. As a result, $B$ is a basis.
    By Lemma \ref{lemma-basis-tie}, $u$ equals $u_B$, the labels with higher scores form $T=\Ext_\rho(B)$ and their number is $p=\ext_\rho(B)$.

    All selected cuts lie strictly inside this block exactly when
    \begin{equation*}
        p<s \quad\text{and}\quad s+\ell-1<p+d=p+n+1.
    \end{equation*}
    These integer inequalities are equivalent to \eqref{eq-contributing-condition}.
    For such a cut $k$, subtract $(k-p)\one_B/d$ from the coordinate hypersimplex in \eqref{eq-score-face}.
    The resulting body is its centered copy $D_{k-p}^{(n)}$ in $E_B$, and the removed vector maps to $(k-p)b_B/d$.
    This proves the face formula \eqref{eq-local-hypersimplex}.

    Conversely, let $B$ satisfy \eqref{eq-contributing-condition}.
    At its unique tie $u_B$, every selected index $k-p$ lies in $\{1,\dots,n\}$.
    Each hypersimplex $D_{k-p}^{(n)}$ is full-dimensional and $T_B$ is an isomorphism by Lemma \ref{lemma-basis-tie}.
    Therefore, all selected faces are full-dimensional in $W$, their mixed volume is positive, and their sum is one of the full-dimensional projected upper cells in Lemma \ref{lemma-upper-cells}.
    Consequently, the corresponding nonzero summand exists.

    To verify the uniqueness, note that a full-dimensional projected cell determines its affine roof and hence the normalized upper normal by \eqref{eq-roof-contact}.
    At that normal, the score partition and the specified cuts recover the block $B$ containing them.
    Therefore, distinct bases satisfying the condition cannot give the same cell or summand.
    Within a basis cell, the maximizing $k$-subsets are $T\cup J$ with $J\subset B$ and $|J|=k-p$.
    If two such subsets have the same projected image, then
    \begin{equation*}
        \sum_{j\in B}(\one_J-\one_{J'})_j b_j=0,\qquad \sum_{j\in B}(\one_J-\one_{J'})_j=0.
    \end{equation*}
    Affine independence of the $b_j$, $j\in B$, implies $J=J'$, so the exposed labeled generators remain distinct inside each contributing cell.
    In particular, if $a_j=a_k$ with $j\neq k$, then the vector $e_j-e_k$ is a circuit vector and
    \begin{equation*}
        s_j(u)-s_k(u)=\rho_j-\rho_k\neq0.
    \end{equation*}
    Therefore, repeated columns cannot lie in the same tied block, and the bijection also holds in this case.
\end{proof}

We now prove the consecutive-multiplicity formula, which is the \emph{structural theorem} of the paper.
It combines the upper-cell formula with Proposition \ref{prop-classification} to express the relevant mixed volumes in terms of the coefficients of $f_A$.
The case where all multiplicities are one gives Theorem \ref{thm-representation}. The two specializations with repeated endpoints in Corollary \ref{cor-three-lists} give the other mixed volumes needed to prove Theorem \ref{thm-main}.

\begin{theorem}\label{thm-structural}
    Let $c=(c_1,\dots,c_\ell)$ be a composition of $n$, with every $c_j\geq1$, and let $1\leq s\leq N-\ell$.
    For $0\leq a\leq n-\ell$, define $q^{(a)}\in\Z_{\geq0}^n$ by
    \begin{equation*}
        q^{(a)}_{a+j}=c_j\quad(1\leq j\leq\ell),\qquad q^{(a)}_k=0\quad\text{otherwise}.
    \end{equation*}
    Then
    \begin{equation}\label{eq-structural}
        \MV_W(P_s[c_1],P_{s+1}[c_2],\dots,P_{s+\ell-1}[c_\ell])=\frac{1}{n!}\sum_{a=0}^{n-\ell} \ME(q^{(a)}) \alpha_{s-a-1}.
    \end{equation}
    The coefficients outside $\{0,\dots,r\}$ are interpreted as zero.
\end{theorem}
\begin{proof}
    Each selected polytope $P_{s+j-1}$ is full-dimensional by Lemma \ref{lemma-projected-properties}, so Lemma \ref{lemma-upper-cells} applies to the list of $n$ lifted entries with the prescribed repetitions.
    By Proposition \ref{prop-classification}, every nonzero summand in \eqref{eq-upper-additivity} corresponds to exactly one basis $B$ satisfying the stated condition.
    Let $p=\ext_\rho(B)$ and $a=s-p-1$.
    Then $0\leq a\leq n-\ell$, and the index of the local hypersimplex at the $j$th selected cut is
    \begin{equation*}
        (s+j-1)-p=a+j.
    \end{equation*}
    The coordinate identification $E_B\cong E_n$ preserves the root-lattice density.
    Translation invariance, the face formula \eqref{eq-local-hypersimplex}, and the mixed-volume change \eqref{eq-TB-mixed-volume} then give
    \begin{equation*}
        \begin{aligned}
            &\MV_W(F_s(u_B)[c_1],\dots,F_{s+\ell-1}(u_B)[c_\ell])\\
            &\qquad=w_A(B) \MV_{E_n}(D_{a+1}^{(n)}[c_1],\dots,D_{a+\ell}^{(n)}[c_\ell])\\
            &\qquad=\frac{w_A(B)}{n!}\ME(q^{(a)}),
        \end{aligned}
    \end{equation*}
    where the last equality is the normalization \eqref{eq-ME}.
    Grouping the nonzero summands first by $a$ and then by basis gives
    \begin{equation*}
        \begin{aligned}
            &\MV_W(P_s[c_1],\dots,P_{s+\ell-1}[c_\ell])\\
            &\qquad=\frac{1}{n!}\sum_{a=0}^{n-\ell} \ME(q^{(a)})\sum_{\substack{B\in\mathcal{B},\\ \ext_\rho(B)=s-a-1}}w_A(B)\\
            &\qquad=\frac{1}{n!}\sum_{a=0}^{n-\ell}\ME(q^{(a)})\alpha_{s-a-1}.
        \end{aligned}
    \end{equation*}
    The inner sum is the coefficient defined in \eqref{eq-polynomial}. Note that it equals zero when $s-a-1\notin\{0,\dots,r\}$.
    This proves \eqref{eq-structural} for the stated range, including both endpoints.
    Lemma \ref{lemma-upper-cells} uses a separate parameter for every occurrence of a lifted polytope, even when occurrences coincide.
    The $n!$ factors in that coefficient comparison cancel on the two sides of \eqref{eq-upper-additivity}.
    The factor $1/n!$ in \eqref{eq-structural} comes from the normalization \eqref{eq-ME}, and no additional factor $c_1!\cdots c_\ell!$ occurs.
\end{proof}

\begin{proof}[Proof of Theorem \ref{thm-representation}]
    In Theorem \ref{thm-structural}, take $\ell=n$, $c=(1,\dots,1)$, and $s=i+1$.
    For $0\leq i\leq r$, we have $1\leq s\leq r+1=N-n$, so these indices satisfy the hypotheses.
    Only $a=0$ occurs, and Lemma \ref{lemma-constants} gives
    \begin{equation*}
        \ME(1,\dots,1)=n!\MV_{E_n}(D_1^{(n)},\dots,D_n^{(n)})=n!.
    \end{equation*}
    Consequently \eqref{eq-structural} becomes
    \begin{equation*}
        \MV_W(P_{i+1},\dots,P_{i+n})=\alpha_i,
    \end{equation*}
    which is \eqref{eq-window-intro}.
    The argument applies to every circuit-generic $\rho$, whereas the displayed mixed volume depends only on the arrangement and its density.
    It therefore also proves independence of $f_{A,\rho}$ from $\rho$ for $d\geq 2$.
    The rank-one verification is given in Section \ref{subsection-quadratic}.
\end{proof}

\section{Coefficient inequalities and growth bounds}
\label{section-inequalities}

\subsection{Three specializations and the quadratic inequality}
\label{subsection-quadratic}

Theorem \ref{thm-structural} gives both the coefficient itself and the two mixed volumes with repeated entries needed for the Alexandrov--Fenchel inequality.
All these mixed volumes are computed in the same $n$-dimensional space $W$ with its induced density.

\begin{corollary}\label{cor-three-lists}
    For every $0\leq i\leq r$,
    \begin{equation}\label{eq-window}
        \MV_W(P_{i+1},\dots,P_{i+n})=\alpha_i.
    \end{equation}
    For $n\geq2$ and the same range of $i$,
    \begin{gather}
        \MV_W(P_{i+1}[2], P_{i+2},\dots,P_{i+n-1})=\frac{\alpha_i+(n-1)\alpha_{i-1}}n,  \label{eq-left-window}\\
        \MV_W(P_{i+2},\dots,P_{i+n-1},P_{i+n}[2])=\frac{\alpha_i+(n-1)\alpha_{i+1}}n.  \label{eq-right-window}
    \end{gather}
    When $n=2$, the unrepeated lists are empty.
\end{corollary}
\begin{proof}
    The first identity was just proved.
    For \eqref{eq-left-window}, apply Theorem \ref{thm-structural} with $s=i+1$ and $c=(2,1,\dots,1)$ of length $n-1$.
    The two possible shifts $a=0, 1$ give the local lists
    \begin{equation*}
        (D_1^{(n)}[2], D_2^{(n)},\dots,D_{n-1}^{(n)}),\quad (D_2^{(n)}[2], D_3^{(n)},\dots,D_n^{(n)}).
    \end{equation*}
    Their mixed volumes in $E_n$ are $1/n$ and $(n-1)/n$, respectively, by Lemma \ref{lemma-constants}.
    After the factor $n!$ in \eqref{eq-ME} cancels the factor $1/n!$ in \eqref{eq-structural}, these multiply $\alpha_i$ and $\alpha_{i-1}$.
    For \eqref{eq-right-window}, use $s=i+2$ and $c=(1,\dots,1,2)$.
    Reflection of the hypersimplices in $E_n$ gives factors $(n-1)/n$ and $1/n$, multiplying $\alpha_{i+1}$ and $\alpha_i$.
    All indices of the projected polytopes lie between $1$ and $N-1$, including when $i=0$ or $i=r$.
    The zero extension of the coefficients includes these two cases.
\end{proof}

\begin{proof}[Proof of Theorem \ref{thm-main}]
    We first consider the rank one case.
    For higher rank, we prove positivity and the degree assertion, then the coefficient symmetry \eqref{eq-palindromicity}, the quadratic inequality, and its consequences for log-concavity, equality, and trapezoidality.

    If $d=1$, all columns are equal to the same nonzero vector.
    The two-label circuits force the $\rho_j$ to be distinct.
    The singleton bases have activities $0,1,\dots,N-1$, once each, and a common positive weight $w$. Therefore,
    \begin{equation}
        \label{eq-rank-one}
        f_A(t)=w(1+t+\dots+t^{N-1}),
    \end{equation}
    which proves all the assertions in rank one including the endpoint cases of \eqref{eq-quadratic} with the prescribed zero extension.
    It also proves the invariance under the circuit-generic choice of $\rho$ in this rank.

    Now suppose $d\geq2$ and let $n=d-1$. We first prove positivity.
    Every entry of the mixed volume in \eqref{eq-window} is full-dimensional by Lemma \ref{lemma-projected-properties}.
    For a fixed $i$, translate these $n$ bodies separately so that each contains the origin in its interior, and choose a full-dimensional ball $B_\varepsilon\subset W$ contained in all of them.
    Translation invariance and monotonicity give
    \begin{equation*}
        \alpha_i=\MV_W(P_{i+1},\dots,P_{i+n})\geq \MV_W(B_\varepsilon[n])=\op{vol}_W(B_\varepsilon)>0.
    \end{equation*}
    Thus every $\alpha_i$, $0\leq i\leq r$, is positive, and $\alpha_r>0$ proves $\deg f_A=r$.

    We next prove the palindromicity in \eqref{eq-palindromicity}.
    Complementing the cube coordinates gives the reflection identity \eqref{eq-reflection}.
    Since $N=r+n+1$, the coefficient formula, translation invariance, simultaneous reflection, and symmetry of mixed volume give
    \begin{align*}
        \alpha_{r-i}&=\MV_W(P_{r-i+1},\dots,P_{r-i+n})\\
        &=\MV_W(b_{[N]}-P_{i+n},\dots,b_{[N]}-P_{i+1})\\
        &=\MV_W(P_{i+1},\dots,P_{i+n})=\alpha_i
    \end{align*}
    for $0\leq i\leq r$. We will use this symmetry to prove trapezoidality.

    To prove the quadratic inequality, suppose that $d\geq 3$, so $n\geq 2$.
    For $0\leq i\leq r$, apply \eqref{eq-AF} in $U=W$, of dimension $m=n$, with
    \begin{equation*}
        K_1=P_{i+1},\quad K_2=P_{i+n},\quad K_j=P_{i+j-1},\quad 3\leq j\leq n.
    \end{equation*}
    The common list $K_3,\dots,K_n$ is empty when $n=2$.
    Corollary \ref{cor-three-lists} gives
    \begin{equation}\label{eq-AF-coefficients}
        \alpha_i^2\geq\frac{\alpha_i+(n-1) \alpha_{i-1}}{n}\, \frac{\alpha_i+(n-1)\alpha_{i+1}}{n},
    \end{equation}
    which reduces to
    \begin{equation*}
        (n+1)\alpha_i^2\geq\alpha_i(\alpha_{i-1}+\alpha_{i+1})+(n-1)\alpha_{i-1}\alpha_{i+1}.
    \end{equation*}
    Since $d=n+1$, this is \eqref{eq-quadratic}, including both endpoint indices.
    For $d=2$, we have $n=1$, and \eqref{eq-window} gives $\alpha_i$ as the length of $P_{i+1}$ in $W$.
    The endpoint polytopes $P_0$ and $P_N$ have length zero.
    Taking lengths in \eqref{eq-midpoint-inclusion} then gives
    \begin{equation*}
        2\alpha_i\geq\alpha_{i-1}+\alpha_{i+1},\quad 0\leq i\leq r.
    \end{equation*}
    Multiplying the positive coefficient $\alpha_i$ proves \eqref{eq-quadratic} in this case as well.

    We now deduce ordinary log-concavity and its equality criterion from \eqref{eq-quadratic}.
    At an internal index, let $a=\alpha_{i-1}$, $b=\alpha_i$, $c=\alpha_{i+1}$, and $x=b/\sqrt{ac}$.
    Positivity makes these quantities well defined.
    The arithmetic--geometric mean inequality and \eqref{eq-quadratic} imply
    \begin{equation*}
        dx^2\geq x\frac{a+c}{\sqrt{ac}}+d-2\geq 2x+d-2.
    \end{equation*}
    Hence $(x-1)(dx+d-2)\geq 0$.
    The second factor is positive, so $x\geq1$ and $b^2\geq ac$.
    If $b^2=ac$, then \eqref{eq-quadratic} reduces to $2ac\geq b(a+c)$.
    Dividing by $b=\sqrt{ac}$ forces equality in the arithmetic--geometric mean inequality, hence $a=b=c$.
    The converse is immediate and proves \eqref{eq-lc-equality}.

    It remains to prove trapezoidality.
    If $r=0$, the sequence consists of one positive term and the assertion is immediate.
    For $r\geq1$, define $R_j=\alpha_j/\alpha_{j-1}$ for $1\leq j\leq r$.
    Log-concavity makes $R_1,\dots,R_r$ nonincreasing, while \eqref{eq-palindromicity} gives
    \begin{equation*}
        R_{r-j+1}=\frac{\alpha_{r-j+1}}{\alpha_{r-j}}=\frac{\alpha_{j-1}}{\alpha_j}=R_j^{-1}.
    \end{equation*}
    Let $q$ be the number of these ratios that are greater than one.
    Monotonicity places them in the first $q$ positions, and the displayed symmetry places exactly $q$ ratios less than one in the last $q$ positions.
    All remaining ratios equal one, and $0\leq q\leq\lfloor r/2\rfloor$.
    For odd $r$, the central ratio is its own reciprocal and hence equals one.
    Consequently,
    \begin{equation*}
        \alpha_0<\cdots<\alpha_q=\cdots=\alpha_{r-q}>\cdots>\alpha_r,
    \end{equation*}
    Therefore, the coefficient sequence has one centered plateau, which completes the proof.
\end{proof}

\subsection{Power concavity}\label{subsection-power}

We deduce Corollary \ref{cor-power} from positivity and the quadratic inequality in Theorem \ref{thm-main}, using the following lemma.
Its proof is independent of flat arrangements and mixed volumes.

\begin{lemma}
    \label{lemma-power-numerical}
    Let $a,b,c>0$ and let $k\geq1$ be an integer. If
    \begin{equation}
    \label{eq-power-input}
        (k+1)b^2\geq b(a+c)+(k-1)ac,
    \end{equation}
    then $2b^{1/k}\geq a^{1/k}+c^{1/k}$. For $k\geq2$, equality in the latter inequality holds if and only if $a=b=c$.
\end{lemma}
\begin{proof}
    When $k=1$ the lemma clearly holds.
    Suppose $k\geq2$. Let $x=(a/b)^{1/k}$, $z=(c/b)^{1/k}$. Dividing \eqref{eq-power-input} by $b^2$ gives
    \begin{equation}
    \label{eq-Fk}
        F_k(x,z)\coloneqq x^k+z^k+(k-1)(xz)^k\leq k+1.
    \end{equation}
    We will show that this implies $x+z\leq2$, where equality holds exactly at $x=z=1$.

    First consider the boundary $x+z=2$. Write $x=1+s$ and $z=1-s$, with $|s|<1$ since $x,z>0$. The binomial formula gives
    \begin{equation*}
        (1+s)^k+(1-s)^k\geq2+k(k-1)s^2.
    \end{equation*}
    For $0\leq t<1$, the function $H_k(t)=(1-t)^k-1+kt$ satisfies
    \begin{equation*}
        H_k(0)=0,\quad H_k'(t)=k(1-(1-t)^{k-1})>0, \quad 0<t<1.
    \end{equation*}
    Taking $t=s^2$ shows that $(1-s^2)^k\geq1-ks^2$, with strict inequality if $s\neq0$. Therefore,
    \begin{equation*}
        F_k(1+s,1-s)\geq 2+k(k-1)s^2+(k-1)(1-ks^2)=k+1,
    \end{equation*}
    where equality holds only for $s=0$.

    If $x+z>2$, let $\lambda=2/(x+z)\in (0,1)$.
    The boundary computation gives
    \begin{equation*}
        k+1\leq F_k(\lambda x,\lambda z)=\lambda^k(x^k+z^k)+(k-1)\lambda^{2k}(xz)^k<F_k(x,z),
    \end{equation*}
    which contradicts \eqref{eq-Fk}, hence $x+z\leq2$, and multiplying by $b^{1/k}$ proves the power inequality.
    If equality holds, then $x+z=2$ and the boundary computation forces $x=z=1$ or $a=b=c$.
    The converse follows by substitution.
\end{proof}

\begin{proof}[Proof of Corollary \ref{cor-power}]
    For an internal index, take $a=\alpha_{i-1}$, $b=\alpha_i$, $c=\alpha_{i+1}$, and $k=d-1$.
    Theorem \ref{thm-main} gives $a,b,c>0$, and its quadratic inequality is exactly \eqref{eq-power-input}.
    Lemma \ref{lemma-power-numerical} therefore proves \eqref{eq-power}, together with its equality criterion when $d\geq3$.

    For $0<p\leq1/k$, let $u_i=\alpha_i^{1/k}$ and $\theta=kp\in (0,1]$.
    The inequality we just proved shows that $(u_i)$ is concave.
    Since $u\mapsto u^\theta$ is increasing and concave on $\R_{\geq0}$, we have
    \begin{equation*}
        \alpha_i^p=u_i^\theta\geq\left(\frac{u_{i-1}+u_{i+1}}2\right)^\theta\geq\frac{u_{i-1}^\theta+u_{i+1}^\theta}2.
    \end{equation*}
    This proves the remaining statement.
\end{proof}

\subsection{Ratio contraction and a binomial growth bound}
\label{subsection-ratios}

Ordinary log-concavity implies that successive coefficient ratios decrease. The rank-sensitive inequality quantifies the decrease. Its endpoint case gives an initial bound on the ratio. Iterating the inequality then gives a polynomial upper bound on coefficient growth.

\begin{proposition}\label{prop-ratios}
    Suppose $d\geq2$ and let $R_i=\alpha_i/\alpha_{i-1}$ for $1\leq i\leq r$.
    For $1\leq i\leq r-1$,
    \begin{equation}\label{eq-ratio-step}
        R_{i+1}\leq\Psi_d(R_i),\quad\Psi_d(x)=\frac{dx-1}{x+d-2}.
    \end{equation}

    If $R_s,R_{s+1},\dots,R_{s+k}$ are all greater than one, then
    \begin{equation}\label{eq-ratio-iterate}
        \frac{1}{R_{s+k}-1}\geq\frac1{R_s-1}+\frac{k}{d-1},\quad R_{s+k}-1\leq\frac{d-1}{k+(d-1)/(R_s-1)}.
    \end{equation}
    For every $1\leq i\leq r$, we also have
    \begin{equation}\label{eq-ratio-global}
        R_i\leq 1+\frac{d-1}{i}.
    \end{equation}
    Consequently, with $m_i=\min(i,r-i)$, we have
    \begin{equation}\label{eq-binomial-bound}
        \alpha_i\leq\alpha_0 \binom{m_i+d-1}{d-1}, \qquad 0\leq i\leq r.
    \end{equation}
\end{proposition}
\begin{proof}
    Let $x=R_i$, $y=R_{i+1}$. Dividing \eqref{eq-quadratic} by $\alpha_i^2$ and multiplying by $x$ gives
    \begin{equation*}
        dx\geq1+(x+d-2)y.
    \end{equation*}
    The denominator $x+d-2$ is positive, which proves \eqref{eq-ratio-step}.
    For $x,y>1$,
    \begin{equation*}
        y-1\leq\Psi_d(x)-1=\frac{(d-1)(x-1)}{x+d-2},
    \end{equation*}
    so taking reciprocals gives
    \begin{equation*}
        \frac{1}{y-1}\geq\frac{1}{x-1}+\frac{1}{d-1}.
    \end{equation*}
    A simple summation proves \eqref{eq-ratio-iterate}.

    The endpoint case of Theorem \ref{thm-main} gives $R_1\leq d$ when $r\geq1$.
    For $x>0$, the function $\Psi_d$ is increasing since
    \begin{equation*}
        \Psi_d'(x)=\frac{(d-1)^2}{(x+d-2)^2}>0.
    \end{equation*}
    A direct substitution gives
    \begin{equation*}
        \Psi_d\left(1+\frac{d-1}{i}\right)=1+\frac{d-1}{i+1}.
    \end{equation*}
    Induction then proves \eqref{eq-ratio-global}, without requiring the ratios to stay above one.
    Multiplying from $1$ to $i$ gives
    \begin{equation*}
        \frac{\alpha_i}{\alpha_0}\leq\prod_{j=1}^i\frac{j+d-1}{j}=\binom{i+d-1}{d-1}.
    \end{equation*}
    Finally, we apply the same bound to $\alpha_{r-i}=\alpha_i$ to get \eqref{eq-binomial-bound}.
\end{proof}

For example, the normalized coefficients in rank two are bounded by $m_i+1$, and those in rank three by $(m_i+1)(m_i+2)/2$.
Therefore, the ratio estimate bounds the coefficient growth away from either endpoint. It also controls the rate at which successive ratios approach one along an increasing part of the sequence.

\section{Consecutive mixed-volume sequences}
\label{section-sequences}

Fix a flat arrangement of rank $d\geq2$. Let $n=d-1$, $c=(c_1,\dots,c_\ell)$ be a composition of $n$ with positive parts.
We consider all admissible placements of these multiplicities along the projected hypersimplices including both endpoints.
Let $p=n-\ell$ and with $q^{(a)}$ as in Theorem \ref{thm-structural}, define
\begin{equation}
    \label{eq-Ec}
    e_a(c)=\ME(q^{(a)}), \qquad E_c(t)=\sum_{a=0}^p e_a(c)t^a.
\end{equation}
We define the corresponding mixed volumes in $W$ by
\begin{equation}\label{eq-Mi}
    M_i(c)=\MV_W(P_{i+1}[c_1],\dots,P_{i+\ell}[c_\ell]),\qquad 0\leq i\leq N-\ell-1,
\end{equation}
and their generating polynomial is $S_c(t)=\sum_iM_i(c)t^i$.
Therefore, $i=0$ begins with $P_1$, and the final placement ends with $P_{N-1}$.

\subsection{The universal mixed-Eulerian factor}

The polynomial $E_c$ is the mixed-Eulerian polynomial of \cite[Definition 7.12]{BST23}.
With $a$ leading zero entries and $p-a$ trailing zero entries in $q^{(a)}$, its generating identity is
\begin{equation}\label{eq-ME-generating}
    E_c(t)=(1-t)^{n+1}\sum_{m\geq0}\prod_{j=1}^{\ell}(m+j)^{c_j}t^m.
\end{equation}
This is \cite[Theorem 7.13]{BST23}, which is also shown independently in \cite[Proposition 4.5]{NT23}. For example, when $n=3$, we have
\begin{equation*}
    E_{(2,1)}(t)=2+4t,\qquad E_{(1,2)}(t)=4+2t.
\end{equation*}

The properties of the universal factor in the following lemma are known.
Positivity is part of the mixed-Eulerian theory, and the value at one is a specialization of Postnikov's cyclic-sum identity \cite[Theorems 16.3--16.4]{Postnikov09}. See also \cite[Section 1.1]{BST23}.
Real rootedness follows from Brenti's criterion \cite[Theorem 4.4.1]{Brenti89} as recalled in \cite[Section 1]{KV26}, and also from \cite[Theorem 1]{KV26}.
We include a proof to fix the normalization and give the one-variable argument needed below.

\begin{lemma}\label{lemma-universal-factor}
    The polynomial $E_c$ has degree $n-\ell$ and positive coefficients, which satisfies $E_c(1)=n!$. If it is nonconstant, all its zeros are negative and real.
\end{lemma}
\begin{proof}
    Each hypersimplex in the mixed volume defining $e_a(c)$ is full-dimensional in $E_n$.
    Therefore, $e_a(c)>0$ for $0\leq a\leq p$, which proves positivity and the degree assertion.
    To compute the value at one, we write the monic polynomial
    \begin{equation*}
        \Phi_c(m)=\prod_{j=1}^{\ell} (m+j)^{c_j}=\sum_{k=0}^n \xi_k \binom{m}{k}.
    \end{equation*}
    Since $\binom{m}{n}$ has leading coefficient $1/n!$, the comparison of leading coefficients gives $\xi_n=n!$.
    The identity
    \begin{equation*}
        \sum_{m\geq0} \binom{m}{k}t^m=\frac{t^k}{(1-t)^{k+1}}
    \end{equation*}
    then gives \eqref{eq-ME-generating} in the form
    \begin{equation}\label{eq-Ec-polynomial}
        E_c(t)=\sum_{k=0}^n \xi_k t^k(1-t)^{n-k}.
    \end{equation}
    At $t=1$, every term with $k<n$ vanishes and the term with $k=n$ is $\xi_n$.
    Therefore $E_c(1)=\xi_n=n!$.

    For real rootedness, consider the universal polynomials for all positive compositions with fixed length $\ell$.
    The initial composition $c=(1,\dots,1)$ has total $\ell$ and gives the constant polynomial $\ell!$.
    Increase its parts one at a time.
    For a current composition of total $m$, multiplying $\Phi_c(q)$ by $q+j$ in the generating series gives
    \begin{equation*}
        \sum_{q\geq 0}\Phi_{c+e_j}(q)t^q=(t\frac{d}{dt}+j)\frac{E_c(t)}{(1-t)^{m+1}}.
    \end{equation*}
    Multiplying by $(1-t)^{m+2}$ gives
    \begin{equation}\label{eq-Ec-recurrence}
        E_{c+e_j}(t)=t(1-t)E_c'(t)+(j+(m+1-j)t) E_c(t),\quad 1\leq j\leq\ell.
    \end{equation}
    This is the differential recurrence underlying \cite[proof of Theorem 7.13]{BST23}. The associated interlacing argument is also a specialization of \cite[Lemma 1]{KV26}.

    Suppose $E_c$ has degree $q=m-\ell$ and positive leading coefficient $a$.
    The leading coefficient on the right of \eqref{eq-Ec-recurrence} is $(\ell+1-j)a>0$, and its constant term is $jE_c(0)>0$.
    For $q=0$, the new polynomial is linear with positive coefficients and has one negative zero.
    For $q>0$, assume inductively that $E_c$ has distinct negative zeros $x_1<\cdots<x_q$.
    At $x_i$, the new polynomial has value $x_i(1-x_i)E_c'(x_i)$ of sign $(-1)^{q-i+1}$, where the signs alternate.
    The sign at $-\infty$ is $(-1)^{q+1}$, opposite to the sign at $x_1$, and the value at zero is positive, opposite to the sign at $x_q$.
    There is therefore a zero in each of
    \begin{equation*}
        (-\infty,x_1),\, (x_1,x_2),\dots,(x_{q-1},x_q),\, (x_q,0).
    \end{equation*}
    The degree is $q+1$, so these are all the zeros and they are distinct and negative.
    Induction then proves the assertion for all positive compositions, including those of total $n=d-1$.
\end{proof}

To check the cited real-rootedness criterion directly in this setting, the zeros of $\Phi_c(z)$ are exactly $-1,-2,\dots,-\ell$, with multiplicities $c_1,\dots,c_\ell$.
They contain every integer between the smallest root and $-1$, as required by Brenti's criterion in the form recalled in \cite[Section 1]{KV26}.
Equivalently, in \cite[Theorem 1]{KV26}, take $a=1$ and $F(z)=\Phi_c(z)/\Phi_c(0)$.
The zeros of $F(-z)$ are $1,\dots,\ell$ and satisfy its condition (15); the numerator in its equation (3) is $E_c(t)/\Phi_c(0)$.
Thus the proof above verifies the known properties of the universal factor directly. We use Theorem \ref{thm-structural} to obtain the following factorization for flat arrangements.

\subsection{Factorization, total sum, and log-concavity}

\begin{theorem}\label{thm-factorization}
    For every positive decomposition $c$ of $n$, we have
    \begin{equation}
        \label{eq-factorization}
        n!S_c(t)=E_c(t)f_A(t).
    \end{equation}
    In particular,
    \begin{equation}
        \label{eq-degrees}
        \op{deg} S_c=r+n-\ell=N-\ell-1,
    \end{equation}
    and
    \begin{equation}
        \label{eq-conservation}
        \sum_{i=0}^{N-\ell-1} M_i(c)=f_A(1)=\sum_{B\in\mathcal{B}} w_A(B).
    \end{equation}
    Therefore, the total sum of the consecutive mixed volumes is independent of $c$.
\end{theorem}
\begin{proof}
    Theorem \ref{thm-structural} with $s=i+1$ gives
    \begin{equation}
        \label{eq-convolution}
        n!M_i(c)=\sum_{a=0}^p e_a(c)\alpha_{i-a}.
    \end{equation}
    Since $r+p=N-\ell-1$, the support of the convolution is exactly the range in \eqref{eq-Mi}.
    Multiplying generating polynomials proves \eqref{eq-factorization}, including both endpoint coefficients.
    Theorem \ref{thm-main} and Lemma \ref{lemma-universal-factor} then give the degree assertion.

    To prove the total sum identity, we evaluate \eqref{eq-factorization} at $t=1$. By Lemma \ref{lemma-universal-factor}, we have $E_c(1)=n!$, and therefore
    \begin{equation*}
        n!\sum_{i=0}^{N-\ell-1} M_i(c)=n!S_c(1)=E_c(1)f_A(1)=n!f_A(1).
    \end{equation*}
    Canceling $n!$ gives the first equality in \eqref{eq-conservation}.
    The defining basis sum \eqref{eq-polynomial} gives $f_A(1)=\sum_B w_A(B)$ since each activity monomial takes value one at $t=1$.
    This proves the second equality and the independence of the total sum from $c$.
\end{proof}

\begin{corollary}\label{cor-sequence-LC}
    For every positive decomposition $c$ of $n$, the sequence in \eqref{eq-Mi} is positive and log-concave:
    \begin{equation}
    \label{eq-sequence-LC}
        M_i(c)^2\geq M_{i-1}(c)M_{i+1}(c), \qquad 1\leq i\leq N-\ell-2.
    \end{equation}
\end{corollary}
\begin{proof}
    The full-dimensionality of every $P_k$ gives $M_i(c)>0$ throughout the range.
    Lemma \ref{lemma-universal-factor} and Newton's inequalities imply that the coefficients of $E_c$ form a log-concave sequence.
    Theorem \ref{thm-main} gives log-concavity of the coefficients of $f_A$.
    Their zero extensions are nonnegative log-concave sequences with interval support.
    Since the convolution preserves this property \cite[Theorem 1.4]{JG06}, \eqref{eq-convolution} then proves \eqref{eq-sequence-LC}.
\end{proof}

\begin{proof}[Proof of Theorem \ref{thm-sequences-intro}]
    The definitions of $E_c$, $M_i(c)$, and $S_c$ in this section agree with those in the introduction by \eqref{eq-ME-generating} and \eqref{eq-Mi}.
    Lemma \ref{lemma-universal-factor} gives the asserted degree and positive coefficients of $E_c$, and shows that its zeros are negative and real when it is nonconstant.
    Theorem \ref{thm-factorization} gives the factorization and the total sum identity, and Corollary \ref{cor-sequence-LC} gives positivity and log-concavity of the entire sequence $(M_i(c))_{i=0}^{N-\ell-1}$.
    This proves Theorem \ref{thm-sequences-intro}.
\end{proof}

\subsection{Relation to earlier formulas}\label{subsection-earlier-formulas}

We first consider the two extreme choices of multiplicities.
For $c=(1,\dots,1)$, the factor is $E_c=n!$, and Theorem \ref{thm-factorization} is exactly the consecutive coefficient representation.
For $c=(n)$, it gives
\begin{equation}
\label{eq-scalar-slices}
    n!\sum_{i=0}^{N-2}\op{vol}_W(P_{i+1})t^i=A_n(t)f_A(t),
\end{equation}
where $A_n(t)=E_{(n)}(t)$ is the ordinary Eulerian polynomial, which is normalized so that the coefficient of $t^a$ counts permutations of $[n]$ with $a$ descents.

The scalar-slice identity \eqref{eq-scalar-slices} is already contained in the zonotopal slice calculation of Galashin--Postnikov--Williams \cite[Lemma 3.11 and the proof of Proposition 4.1]{GPW22}.
To match conventions, we choose coordinates in $V$ that carry $\varphi$ to the last coordinate and $|\omega_V|$ to the coordinate density.
Then their horizontal slice $Q_k$ is a translate of $P_k$ with the same induced volume.
For the regular tiling determined by $\rho$, their Lemma 5.13 and Proposition 6.2 identify the tile of a basis $B$ with $\Pi_{\Ext_\rho(B),B}$.
Thus their $\gamma_j$ is exactly $\alpha_j$ in the present sign convention.
Let $\genfrac{\langle}{\rangle}{0 pt}{}{n}{a}$ denote an Eulerian number, their tile-slice formula consequently reads
\begin{equation}
\label{eq-GPW-slice}
    \op{vol}_W(P_{i+1})=\frac{1}{n!}\sum_{a=0}^{n-1}\genfrac{\langle}{\rangle}{0pt}{}{n}{a}\alpha_{i-a}.
\end{equation}
The factor $1/n!$ comes from their equation (3.9), with ambient dimension $d=n+1$.
Multiplying generating polynomials gives \eqref{eq-scalar-slices} with no change of variable or reversal of the coefficients.
The log-concavity of these single-body volumes also follows directly from \eqref{eq-midpoint-inclusion} and the Brunn--Minkowski inequality \cite[Section 7.1]{Schneider14}.

For an arbitrary positive composition, \eqref{eq-factorization} is the extension from slice volumes to consecutive mixed volumes supplied by Theorem \ref{thm-structural}.
The universal mixed-Eulerian identities remain those of \cite{BST23,NT23}.
For comparison, \cite[Theorem 10.2]{BST23} concerns a loopless matroid $M$ of rank $n+1$ on a ground set of size $N$ and its hypersimplex classes $\gamma_j$ in the matroid Chow ring.
In our dimension and index notation, their formula is
\begin{equation*}
    \sum_{i=0}^{N-\ell-1}\deg_M(\gamma_{i+1}^{c_1}\cdots\gamma_{i+\ell}^{c_\ell})t^i=E_c(t)T_M(1,t),
\end{equation*}
where $\deg_M$ is the degree map on products of total degree $n$.
The present formula instead concerns $n!$ times actual mixed volumes of projected polytopes and the determinant-weighted external semi-activity polynomial $f_A$.
Although the two factorizations have the same form, the two enumerators need not agree.
Our formula applies to non-unimodular flat real arrangements. The conclusion about the coefficient shape then follows by Theorem \ref{thm-main}.

\section{Eulerian digraphs and positive circulation weights}
\label{section-graphs}

\subsection{Graph polynomials and cographic arrangements}

We start by recalling some basic notions about digraphs, following \cite[Sections 1.2 and 1.6--1.7]{BG09}. All graphs in this section are finite directed multigraphs. The underlying undirected multigraph is obtained by forgetting directions but keeping parallel arcs as distinct edges. Connectedness refers to this underlying multigraph.
A connected digraph is \emph{Eulerian} if the in-degree and out-degree agree at every vertex.
Directed loops occur in no spanning tree and contribute equally to both degrees, so they are deleted before the cycle rank is defined.
For a loopless connected digraph $D$, let
\begin{equation*}
    \nu=|V(D)|,\quad m=|E(D)|,\quad\beta(D)=m-\nu+1.
\end{equation*}

A spanning tree means a spanning tree of the underlying labeled undirected multigraph. Once a root $u_0$ is fixed, there is a unique orientation of that tree toward $u_0$.
Let $\kappa_{u_0}(T)$ be the number of arcs whose given orientations disagree with it.
For an Eulerian digraph, the Murasugi--Stoimenow polynomial is given by
\begin{equation}
\label{eq-unweighted-poly}
    P_D(t)=\sum_T t^{\kappa_{u_0}(T)}=\sum_{i=0}^{\nu-1} c_i(D)t^i,
\end{equation}
which is independent of $u_0$ \cite[Proposition 1]{MS03}.
The weighted argument below will also prove root independence directly.
For the trivial digraph with one vertex and no edges, the empty tree contributes $1$.

We now recall graphic and cographic matroids from \cite[Chapters 1--2]{Oxl11}, using the oriented-matroid conventions of \cite{BLSWZ99}.
The graphic matroid of the underlying graph has the forests as its independent sets and the spanning trees as its bases. Its \emph{oriented graphic matroid} is represented by the signed incidence matrix $\partial$, whose column for an arc has $-1$ at its tail and $+1$ at its head.
Equivalently, a signed circuit records a traversal of an undirected cycle, with signs indicating agreement or disagreement with the given arc directions.
The cographic matroid is its dual: its bases are the complements of spanning trees, and its signed circuits are the signed minimal cuts.
A matrix $C$ whose rows form a basis of $\ker\partial$ represents the oriented cographic matroid.
The representations used here may be chosen totally unimodular, meaning that every square minor is $0$, $1$, or $-1$.

More explicitly, we choose a spanning tree $T_0$ and delete one row of $\partial$ to get $\partial'$.
After ordering the tree arcs first and writing $S_0=E(D)\setminus T_0$, the matrix
\begin{equation}
\label{eq-cographic-matrix}
    C=[-((\partial'_{T_0})^{-1}\partial'_{S_0})^T \mid I_{\beta(D)}]
\end{equation}
has row space $\ker\partial$ and is the standard fundamental cycle representation.
The standard dual-representation construction is given in \cite[Theorem 2.2.8]{Oxl11}, where its total unimodularity is discussed in \cite[p. 139]{Oxl11}.
In the notation of the graph-polynomial literature, this is the cographic matrix of \cite[Definition 4.8 and Theorem 4.9]{KMP25}, up to column ordering and a choice of signs for the row basis. The maximal nonzero minors have absolute value one.
We use the standard coordinate density on the $\beta(D)$-dimensional space containing its columns.

\begin{corollary}\label{cor-Eulerian}
    Let $D$ be a connected Eulerian digraph with loops deleted, and let $\beta=\beta(D)$. The coefficients of $P_D$ are positive and palindromic. For $\beta\geq1$,
    \begin{equation}
        \label{eq-Eulerian-quadratic}
        \beta c_i(D)^2\geq c_i(D)(c_{i-1}(D)+c_{i+1}(D))+(\beta-2)c_{i-1}(D)c_{i+1}(D)
    \end{equation}
    at every internal index $1\leq i\leq\nu-2$.
    They are log-concave, where equality holds exactly at a locally constant triple.
    For $\beta\geq2$, their $1/(\beta-1)$-st powers are concave, where equality holds only at a locally constant triple when $\beta\geq3$.
    For $\beta=1$, all coefficients are equal. For $\beta=0$, $P_D=1$.
\end{corollary}
\begin{proof}
    For $\beta\geq1$, let $C$ be \eqref{eq-cographic-matrix}.
    Eulerian balance gives $\one\in\ker\partial=\op{rowspan}(C)$, so a linear functional takes value one on every column of $C$.
    Thus its columns form a flat arrangement of rank $\beta$.
    Its basis weights are all one, and the cographic activity identity gives
    \begin{equation}
        \label{eq-unweighted-bridge}
        P_D(t)=f_C(t)
    \end{equation}
    by \cite[Theorem 5.3]{KMP25}. The corank is $m-\beta=\nu-1$.
    We apply Theorem \ref{thm-main} with $d=\beta$ and $r=\nu-1$, and Corollary \ref{cor-power} when $\beta\geq2$.
    This gives all assertions, including the rank-one case from \eqref{eq-rank-one}.
    A nontrivial connected loopless Eulerian digraph has at least one outgoing arc at each vertex, hence $m\geq\nu$ and $\beta\geq1$.
    Therefore, $\beta=0$ is exactly the trivial case of one vertex without edges.
\end{proof}

Corollary \ref{cor-Eulerian} proves \cite[Conjecture 1.6]{HMV25}.
For $\beta=2$, its power-concavity conclusion is the ordinary additive inequality $2c_i\geq c_{i-1}+c_{i+1}$.
We use determinant weights to extend this result to real circulation weights.

\subsection{Weighted spanning trees and the Laplacian}

Throughout this subsection and the next, $D$ is a connected directed multigraph equipped with positive real arc weights $w_e$ satisfying the circulation equations \eqref{eq-balance-intro}.
Its unweighted in-degrees and out-degrees need not agree.
We use the weighted tree polynomial defined in \eqref{eq-PDw-intro}.
Delete loops and merge each class of arcs with the same ordered endpoints into a single arc whose weight is the sum of the class weights.
Denote the resulting weighted digraph by $D_{\red}$, with arc set $E_{\red}$, and set
\begin{equation*}
    m_{\red}=|E_{\red}|,\quad \beta_{\red}=m_{\red}-\nu+1.
\end{equation*}
The tree polynomial is unchanged: a tree uses at most one arc from a same-oriented parallel class, its reversal statistic is independent of the choice in that class, and summing the possible choices replaces their weights by their sum.
Anti-parallel arcs are not merged since they have different reversal statistics.
This reduction preserves the circulation equations.

For distinct vertices $u,v$, define the weighted out-Laplacian by
\begin{equation}
    \label{eq-Laplacian}
    (L_w)_{uv}=-\sum_{e:u\to v}w_e,\quad (L_w)_{uu}=\sum_{\op{tail}(e)=u}w_e,
\end{equation}
where all loops have already been deleted.
The row sums are zero by construction and the column sums are zero by circulation balance.
Let $\overline{L}_w$ be the principal submatrix obtained by deleting the row and column of a root $u_0$.

For $t>0$, form a weighted digraph $\widehat{D}_w(t)$ as follows.
For each labeled arc $e:u\to v$, keep a forward copy $e^+:u\to v$ of weight $w_e$ and add a reverse copy $e^-:v\to u$ of weight $tw_e$.
The copies arising from different original arcs remain distinct, including when the original arcs are anti-parallel.
We call $\widehat{D}_w(t)$ the doubled digraph.
An \emph{in-arborescence} rooted at $u_0$ is a spanning tree whose arcs point toward $u_0$.

\begin{lemma}\label{lemma-doubled-laplacian}
    The out-Laplacian of $\widehat{D}_w(t)$ is $L_w+tL_w^T$. Its row sums and column sums are zero.
\end{lemma}
\begin{proof}
    Let $d_w^+(u)=\sum_{\op{tail}(e)=u}w_e$ and $d_w^-(u)=\sum_{\op{head}(e)=u}w_e$.
    For distinct $u,v$, the corresponding entry of the doubled out-Laplacian is
    \begin{equation*}
        -\sum_{e:u\to v}w_e-t\sum_{e:v\to u}w_e=(L_w)_{uv}+t(L_w)_{vu}.
    \end{equation*}
    Its diagonal entry at $u$ is $d_w^+(u)+td_w^-(u)$.
    The circulation balance gives $d_w^-(u)=d_w^+(u)$, so this diagonal entry equals $(1+t)(L_w)_{uu}$.
    This proves the matrix identity and shows where circulation balance is used.
    Both $L_w$ and $L_w^T$ have zero row and column sums, so the same holds for their sum above.
\end{proof}

The next formula is a weighted application of the directed Matrix-Tree theorem \cite{Chaiken82}. The unweighted Eulerian case is given in \cite[Lemma 5.1]{HMV25} and attributed there to Murasugi and Stoimenow \cite{MS03}.

\begin{lemma}\label{lemma-matrix-tree}
    The polynomial
    \begin{equation}
        \label{eq-weighted-det}
        P_{D,w}(t)=\det(\overline{L}_w+t\overline{L}_w^{\,T})
    \end{equation}
    is independent of the root and equals
    \begin{equation}
        \label{eq-weighted-trees}
        P_{D,w}(t)=\sum_T\left(\prod_{e\in T}w_e\right)t^{\kappa_{u_0}(T)}.
    \end{equation}
    Here the trees may be counted either before reduction or with the aggregated weights after reduction.
    The polynomial satisfies $t^{\nu-1}P_{D,w}(t^{-1})=P_{D,w}(t)$.
\end{lemma}
\begin{proof}
    Suppose first that $\nu\geq2$ and $t>0$.
    By Lemma \ref{lemma-doubled-laplacian} and the directed Matrix-Tree theorem, the principal cofactor in \eqref{eq-weighted-det} is the total weight of the in-arborescences of $\widehat{D}_w(t)$ rooted at $u_0$.
    There is a bijection between these in-arborescences and the labeled spanning trees of the original underlying multigraph.
    Indeed, an in-arborescence cannot use both copies of one original arc, since they would form a directed two-cycle. Forgetting the copy signs therefore gives a spanning tree.
    Conversely, orienting a fixed labeled tree $T$ toward $u_0$ selects exactly one copy of each arc in $T$.
    The selected copy contributes $w_e$ when the original arc agrees with this orientation and $t w_e$ otherwise.
    Therefore, the weight is $(\prod_{e\in T} w_e)t^{\kappa_{u_0}(T)}$. This proves \eqref{eq-weighted-trees} for $t>0$, and hence as a polynomial identity.

    For $t>0$, the doubled digraph is strongly connected: every undirected path in the connected support can be followed in either direction.
    Its Laplacian has rank $\nu-1$, since its rooted-tree cofactors are positive, and its left and right kernels are both spanned by $\one$.
    Its adjugate is therefore a scalar multiple of $\one\one^T$.
    All principal cofactors agree for $t>0$, so they agree as polynomials.
    This proves root independence.
    Reduction does not change either the Laplacian or the tree sum, as explained above.
    Finally, we have
    \begin{equation*}
        t^{\nu-1}P_{D,w}(t^{-1})=\det(t\overline{L}_w+\overline{L}_w^{T})=\det(\overline{L}_w+t\overline{L}_w^{T}),
    \end{equation*}
    by invariance of determinant under transpose.
    For one vertex, we use the convention that the determinant of the empty matrix is one.
\end{proof}

\subsection{The flat arrangement for real weights}

We next identify the weighted tree polynomial with a flat-arrangement polynomial.
The argument uses an arc order for which the least arc crossing any cut points away from the shore containing the root.
We first construct this order and then relate cographic semi-activity to tree reversals.

\begin{lemma}\label{lemma-rooted-cut-order}
    Let $D$ be a connected loopless digraph with a positive circulation, and fix a vertex $u_0$.
    There is a total order on its arcs such that, for every nonempty proper subset $R\subset V(D)$ containing $u_0$, the least arc between $R$ and its complement leaves $R$.
\end{lemma}
\begin{proof}
    Let $U$ be the set of vertices reachable from $u_0$ by directed paths. Note that no arc leaves $U$.
    Summing the circulation equations over $U$ shows that the total weight entering $U$ is also zero.
    Since all arc weights are positive, no arc enters $U$, and connectedness implies $U=V(D)$.
    For each vertex other than $u_0$, choose the last arc of a shortest directed path from $u_0$ to that vertex.
    These chosen arcs form a spanning tree $Q$ oriented away from $u_0$: the distance from $u_0$ increases by one along each chosen arc, and following the chosen predecessors leads back to $u_0$.

    We order the arcs of $Q$ so that every arc precedes all its descendants and place all other arcs after those of $Q$.
    Every nontrivial cut meets $Q$, so its least arc $e$ belongs to $Q$.
    If $e$ entered $R$, its tail would lie outside $R$, and the directed path in $Q$ from $u_0$ to that tail would contain an earlier arc leaving $R$.
    This contradicts the choice of $e$.
    Therefore, the least crossing arc leaves $R$.
\end{proof}

For Eulerian digraphs, \cite[Lemma 5.2]{KMP25} obtains this cut property from an Euler tour.
Lemma \ref{lemma-rooted-cut-order} gives it directly for any support carrying a positive circulation.
The following activity calculation is the fundamental-cut argument of \cite[Lemma 5.1]{KMP25}, written in the circuit-vector convention of \eqref{eq-fundamental}.

\begin{lemma}\label{lemma-tree-activity}
    Let $C$ be the cographic matrix \eqref{eq-cographic-matrix} of a loopless connected digraph $D$ with $m$ arcs. Fix a root $u_0$ and an order $e_1<\cdots<e_m$ with the cut property of Lemma \ref{lemma-rooted-cut-order}. Let $\rho_{e_k}=2^{m-k}$. Then $\rho$ is circuit-generic for $C$, and every spanning tree $T$, with complementary cographic basis $S=E(D)\setminus T$, satisfies
    \begin{equation}
        \label{eq-weighted-activity}
        \ext_\rho(S)=\kappa_{u_0}(T).
    \end{equation}
\end{lemma}
\begin{proof}
    Every circuit vector of $C$ is a nonzero scalar multiple of a signed minimal-cut vector, whose nonzero entries are $1$ and $-1$.
    The inequality $2^{m-k}>\sum_{\ell>k}2^{m-\ell}$ shows that the pairing with $\rho$ has the sign of the first nonzero coordinate of any such circuit vector.
    In particular, the pairing is nonzero, so $\rho$ is circuit-generic.

    Fix $j\in T$, and let $R$ be the component of $T\backslash\{j\}$ containing $u_0$.
    The cut vector $q=\partial^T\one_R$ lies in $\ker C$ because $C\partial^T=0$.
    Its support is contained in $S\cup\{j\}$, since $j$ is the only tree arc crossing the cut and $q_j\in\{1,-1\}$.
    As $S$ is a basis, the normalized fundamental circuit is therefore
    \begin{equation*}
        z^{S,j}=q/q_j.
    \end{equation*}
    Let $e_*$ be the least crossing arc.
    It leaves $R$, so $q_{e_*}=-1$ by the incidence convention.
    Consequently $\langle\rho,z^{S,j}\rangle>0$ if and only if $q_j=-1$, which holds exactly when $j$ leaves $R$.
    Such an arc must be reversed to point toward $u_0$. Counting these arcs $j\in T$ proves \eqref{eq-weighted-activity}.
\end{proof}

We apply these lemmas to the reduced weighted digraph $D_{\red}$.
We include the numerical weights by rescaling the columns by positive factors and computing the resulting determinant weights.

\begin{proposition}\label{prop-weighted-bridge}
    Assume $\nu\geq2$ and let $C=[c_e]_{e\in E_{\red}}$ be the totally unimodular cographic matrix \eqref{eq-cographic-matrix} of $D_{\red}$.
    Let
    \begin{equation}
        \label{eq-weighted-arrangement}
        D_w=\op{diag}(w_e:e\in E_{\red}),\quad\widetilde{C}=CD_w^{-1}.
    \end{equation}
    Then $\widetilde{C}$ is a flat arrangement of rank $\beta_{\red}$ and
    \begin{equation}
        \label{eq-weighted-bridge}
        P_{D,w}(t)=\left(\prod_{e\in E_{\red}}w_e\right)f_{\widetilde{C}}(t).
    \end{equation}
\end{proposition}
\begin{proof}
    The vector $w=(w_e)$ lies in $\ker\partial=\op{rowspan}(C)$, so there is a functional $\varphi_w$ on the column space of $C$ with $\varphi_w(c_e)=w_e$.
    After scaling, $\varphi_w(c_e/w_e)=1$, so $\widetilde{C}$ is a flat arrangement.
    Its rank is $\beta_{\red}$ since the scaling matrix is invertible.

    Choose an arc order by Lemma \ref{lemma-rooted-cut-order} and the circuit-generic vector $\rho$ of Lemma \ref{lemma-tree-activity}.
    Fix a spanning tree $T$, let $S=E_{\red}\setminus T$ and $j\in T$.
    Let $z^{S,j}$ be the normalized fundamental circuit vector for $C$.
    For $\widetilde{C}=CD_w^{-1}$, the corresponding normalized vector is
    \begin{equation}
        \label{eq-scaled-circuit}
        \widetilde{z}^{S,j}=\frac{D_wz^{S,j}}{w_j}.
    \end{equation}
    Taking $\widetilde{\rho}=D_w^{-1}\rho$ then gives
    \begin{equation}
        \label{eq-scaled-activity}
        \langle\widetilde{\rho},\widetilde{z}^{S,j}\rangle=\frac{\langle\rho,z^{S,j}\rangle}{w_j}.
    \end{equation}
    The positive diagonal scaling preserves minimal dependence supports, so it carries all circuits to circuits and makes $\widetilde{\rho}$ circuit-generic for $\widetilde{C}$.
    Since $w_j>0$, every activity sign is preserved, hence \eqref{eq-weighted-activity} holds for the scaled arrangement as well.

    Finally, the total unimodularity gives $|\det C_S|=1$ for every cographic basis.
    With the standard coordinate density, we have
    \begin{equation}
        \label{eq-scaled-determinant}
        |\det\widetilde{C}_S|=\prod_{e\in S}w_e^{-1},\quad\left(\prod_{e\in E_{\red}}w_e\right)|\det\widetilde{C}_S|=\prod_{e\in T}w_e.
    \end{equation}
    Summing \eqref{eq-polynomial} over complements of spanning trees, together with \eqref{eq-weighted-activity} and \eqref{eq-scaled-determinant}, then proves \eqref{eq-weighted-bridge}.
\end{proof}

We can now prove the main result of this section.
The following theorem strengthens Theorem \ref{thm-weighted-intro} by including the equality cases for power concavity.

\begin{theorem}
    \label{thm-weighted}
    For a connected directed multigraph with positive real circulation weights, the polynomial
    \begin{equation*}
        P_{D,w}(t)=\sum_{i=0}^{\nu-1}c_i(D,w)t^i
    \end{equation*}
    is independent of the root and has positive, palindromic coefficients.
    For $\beta_{\red}\geq1$ and every $1\leq i\leq\nu-2$,
    \begin{equation}
        \label{eq-weighted-quadratic}
        \beta_{\red}c_i(D,w)^2\geq c_i(D,w)(c_{i-1}(D,w)+c_{i+1}(D,w))+(\beta_{\red}-2)c_{i-1}(D,w)c_{i+1}(D,w).
    \end{equation}
    The coefficients are log-concave, with equality exactly at a locally constant triple, and hence are trapezoidal.
    For $\beta_{\red}\geq2$ and every internal index $1\leq i\leq\nu-2$,
    \begin{equation}
        \label{eq-weighted-power}
        2c_i(D,w)^{1/(\beta_{\red}-1)}\geq c_{i-1}(D,w)^{1/(\beta_{\red}-1)}+c_{i+1}(D,w)^{1/(\beta_{\red}-1)}.
    \end{equation}
    For $\beta_{\red}\geq3$, equality in \eqref{eq-weighted-power} occurs exactly at a locally constant triple.
    For $\beta_{\red}=1$, the coefficient sequence is constant. For $\beta_{\red}=2$, \eqref{eq-weighted-power} is ordinary additive concavity.
    For $\beta_{\red}=0$, the polynomial is $1$.
\end{theorem}
\begin{proof}
    If $\nu\geq2$, every vertex has a positive outgoing weight: otherwise balance will give no incident arcs, which contradicts connectedness.
    Thus the reduced support has at least $\nu$ arcs and $\beta_{\red}\geq1$.
    Proposition \ref{prop-weighted-bridge} identifies $P_{D,w}$, up to a positive scalar, with the polynomial of a flat arrangement of rank $\beta_{\red}$ and corank
    \begin{equation*}
        |E_{\red}|-\beta_{\red}=\nu-1.
    \end{equation*}
    Theorem \ref{thm-main} with $d=\beta_{\red}$ gives the coefficient inequality, positivity, symmetry, and the ordinary log-concavity equality statement.
    Corollary \ref{cor-power} gives the power-concavity statements when $\beta_{\red}\geq2$.
    All these conclusions are unchanged by common positive rescaling of the coefficients.
    For $\nu=1$, loop deletion gives the edgeless graph and the polynomial $1$.
    Root independence and the determinant formula follow from Lemma \ref{lemma-matrix-tree}.
\end{proof}

\begin{proof}[Proof of Theorem \ref{thm-weighted-intro}]
    Theorem \ref{thm-weighted} gives root independence, positivity, palindromicity, and log-concavity for $P_{D,w}$, together with the stated characterization of equality in log-concavity. If $D$ is Eulerian, the weights $w_e=1$ satisfy the circulation equations \eqref{eq-balance-intro}, since the in-degree and out-degree agree at every vertex. For these weights, $P_{D,w}=P_D$, so the same conclusions hold for the Murasugi--Stoimenow polynomial.
\end{proof}

\begin{example}\label{example-nonEulerian-support}
    The unweighted support of a positive circulation need not be Eulerian.
    On vertices $1,2,3$, take arcs $1\to 2$, $2\to 1$, $2\to 3$, and $3\to 1$, with respective weights $a+b,a,b,b$, where $a,b>0$ are arbitrary real numbers.
    The weighted degrees balance at every vertex, but the unweighted in-degree and out-degree differ at vertices $1$ and $2$.
    Here $\beta_{\red}=2$, and deleting the row and column of vertex $1$ gives
    \begin{equation*}
        \overline{L}_w=\begin{pmatrix}a+b&-b \\ 0&b\end{pmatrix},\qquad P_{D,w}(t)=b(a+b)+b(2a+b)t+b(a+b)t^2.
    \end{equation*}
    Theorem \ref{thm-weighted} gives additive concavity of these coefficients.
    Therefore, the weighted theorem also applies to supports that are not Eulerian, provided the weights satisfy circulation balance.
\end{example}

Merging parallel arcs with the same orientation can reduce the rank parameter without changing the polynomial.
Therefore, the weighted theorem can give stronger estimates even for an originally unweighted multigraph.
The ratio and coefficient-growth bounds of Proposition \ref{prop-ratios} transfer through the same weighted identity.

\section{Refined inequalities for special alternating links}\label{section-links}

In this section, we apply the Eulerian-digraph inequalities of Section \ref{section-graphs} to the Alexander polynomials of special alternating links.
We obtain a quadratic coefficient inequality and power concavity with an exponent determined by the number of Seifert circles, together with their equality cases.
The passage from link diagrams to Eulerian digraphs uses the checkerboard construction and the Alexander-polynomial identity of Murasugi and Stoimenow.

We first recall some terminology following \cite[Sections 2.1, 4.5, and 5.1]{Murasugi96} and \cite[Section 4.2]{MS03}.
An oriented link diagram is a planar projection of an oriented link with only transverse double points, together with overcrossing and undercrossing information at each double point.
It is connected if its underlying projection is connected, and alternating if overcrossings and undercrossings alternate along each component.
A crossing is nugatory if a simple closed curve meets the projection transversely at that crossing and nowhere else. A diagram is reduced if it has no nugatory crossings.
The oriented smoothing replaces a crossing by two disjoint arcs compatible with the orientations. Smoothing every crossing produces the Seifert circles, and Seifert's construction fills these circles by disks and restores the crossings by half-twisted bands.

For special alternating diagrams we use the plane bipartite graph construction of \cite[Section 5, Proposition 4]{MS03}. See also \cite[Section 1.2]{HMV25} for the orientation convention.
Here a plane graph means a graph with a fixed embedding in the plane, and parallel edges are kept.
Given a connected plane bipartite graph $G$ with at least one edge, form its medial projection by placing a crossing at each edge midpoint and joining successive midpoints around the vertices of $G$. The regions corresponding to vertices of one bipartition class have counterclockwise oriented boundaries, and those corresponding to the other class have clockwise oriented boundaries.
Choose the overcrossings and undercrossings such that every crossing is positive. This gives the positive special alternating diagram $L_G$.
Its mirror is the negative special alternating diagram.
Equivalently, the Seifert surface consists of one disk at each vertex of $G$ and one crossing band along each edge.
The crossingless unknot is included separately, corresponding to a single vertex and no edges. A special alternating link here is an oriented link admitting one of these diagrams.

Let $\mathcal{D}$ be a connected reduced special alternating diagram of $L$, and write $c$ for its number of crossings and $s$ for its number of Seifert circles.
We use the one-variable Alexander polynomial with integral exponents, defined up to multiplication by a unit $\pm t^k$. See \cite[Section 4.1]{MS03}.
We write $X\eqdot Y$ for equality up to such a unit and choose the normalization
\begin{equation}
    \label{eq-link-normalization}
    F_L(t)\eqdot\Delta_L(-t),
\end{equation}
with lowest exponent zero and positive coefficients.
The proof below shows that this normalization exists and determines its degree.

\begin{theorem}\label{thm-links}
    Let $L$ admit a connected reduced special alternating diagram with $c$ crossings and $s$ Seifert circles. Its normalized Alexander polynomial is
    \begin{equation*}
        F_L(t)=\sum_{i=0}^{c-s+1}\gamma_i t^i,
    \end{equation*}
    with every $\gamma_i>0$ and $\gamma_i=\gamma_{c-s+1-i}$.
    At every internal index,
    \begin{equation}
        \label{eq-link-quadratic}
        (s-1)\gamma_i^2\geq\gamma_i(\gamma_{i-1}+\gamma_{i+1})+(s-3)\gamma_{i-1}\gamma_{i+1},\qquad 1\leq i\leq c-s.
    \end{equation}
    The sequence is log-concave, with equality exactly at a locally constant triple, and is trapezoidal.
    For $s\geq3$ and every internal index,
    \begin{equation}
        \label{eq-link-power}
        2\gamma_i^{1/(s-2)}\geq\gamma_{i-1}^{1/(s-2)}+\gamma_{i+1}^{1/(s-2)}.
    \end{equation}
    For $s\geq4$, the equality in \eqref{eq-link-power} holds exactly at a locally constant triple.
    For $s=1$, the diagram is the crossingless unknot and $F_L=1$.
    For $s=2$,
    \begin{equation*}
        F_L(t)=[c]_t=1+t+\cdots+t^{c-1}.
    \end{equation*}
\end{theorem}
\begin{proof}
    Taking the mirror preserves $c$ and $s$ and changes the one-variable Alexander polynomial only by a unit, by the Seifert-matrix description in \cite[Proposition 5.4.7 and Section 6.1]{Murasugi96}, thus the normalized coefficient sequence is unchanged.
    We may assume that $\mathcal{D}=L_G$ is positive special alternating.

    In the disk and band description, each edge of $G$ contributes exactly one crossing. At that crossing, the oriented smoothing reconnects the boundary arcs belonging to each of the two incident vertex disks.
    After every crossing is smoothed, these arcs form exactly the boundaries of the individual disks, thus the Seifert circles correspond bijectively to the vertices of $G$, and the crossings correspond bijectively to its edges:
    \begin{equation}
        \label{eq-Seifert-counts}
        |V(G)|=s,\qquad |E(G)|=c.
    \end{equation}
    This is the typical graph whose vertices correspond to Seifert-circle regions in \cite[Section 5, Proposition 4]{MS03}.
    If $s=1$, bipartiteness forbids edges, so the diagram is the crossingless unknot and $F_L=1$. Now suppose $s\geq2$.

    The graph $G$ is loopless since it is bipartite, and it is bridgeless since $\mathcal{D}$ is reduced.
    To see the latter implication, note that a bridge has the same face on both sides. Therefore, the two corresponding opposite regions at the medial crossing belong to one region.
    An arc in that region joining these opposite sectors, closed through the crossing, gives a simple closed curve meeting the projection only at the crossing.
    This would make the crossing nugatory. 
    See \cite[Section 4.2, Remark 4]{MS03}.

    Let $D$ be the plane dual of $G$, directing each dual edge so that the endpoint of the crossed edge in one fixed bipartition class lies to its left and the other endpoint lies to its right.
    Along each face boundary of $G$, successive vertices belong to opposite bipartition classes.
    The incident dual edges therefore alternate between incoming and outgoing around each vertex of $D$.
    In particular, $D$ is a connected Eulerian digraph. An embedding with this alternation property is called an alternating dimap \cite[Section 1.1]{HMV25}.
    A dual edge is a loop precisely when its primal edge is a bridge, so $D$ is loopless.
    The edges of $D$ correspond bijectively to the edges of $G$, and the vertices of $D$ correspond to the faces of $G$.
    Therefore, Euler's formula and \eqref{eq-Seifert-counts} give
    \begin{equation*}
        |V(D)|=c-s+2,\qquad|E(D)|=c,
    \end{equation*}
    and hence
    \begin{equation}\label{eq-link-rank}
        \beta(D)=|E(D)|-|V(D)|+1=s-1,\qquad|V(D)|-1=c-s+1.
    \end{equation}

    The Alexander-polynomial identity of \cite[Theorem 2]{MS03} is
    \begin{equation}
        \label{eq-link-bridge}
        \Delta_L(-t)\eqdot P_D(t).
    \end{equation}
    Its spanning-tree exponent is the number of edges that must be reversed to point toward the root, exactly as in \eqref{eq-unweighted-poly}. See also \cite[Theorem 1.4]{HMV25}.
    Corollary \ref{cor-Eulerian} shows that $P_D$ has positive coefficients in every degree from zero to $|V(D)|-1$.
    The nonzero constant coefficient determines the monomial factor in \eqref{eq-link-bridge}, and positivity determines its sign.
    Therefore, $F_L=P_D$, and \eqref{eq-link-rank} gives the degree and identifies the coefficient sequences.
    Applying Corollary \ref{cor-Eulerian} with $\beta(D)=s-1$ proves \eqref{eq-link-quadratic}, palindromicity, and the log-concavity equality statement.
    Positivity, palindromicity, and that equality statement then imply trapezoidality.
    For $s\geq3$, the same corollary gives \eqref{eq-link-power} with exponent $1/(\beta(D)-1)=1/(s-2)$. When $s\geq4$, its equality criterion applies because $\beta(D)\geq3$.

    If $s=2$, then $G$ consists of two vertices joined by $c$ parallel edges, with $c\geq 2$ by reducedness.
    Its dual is a coherently directed $c$-cycle.
    Write its vertices as $v_0,\dots,v_{c-1}$ in cyclic order and root it at $v_0$.
    For $0\leq j\leq c-1$, deleting the edge $v_j\to v_{j+1}$ with indices read modulo $c$ gives a spanning tree in which exactly the $j$ edges on the path from $v_0$ to $v_j$ must be reversed.
    These are all the spanning trees, so $P_D(t)=1+t+\dots+t^{c-1}$.
\end{proof}

\begin{proof}[Proof of Theorem \ref{thm-links-intro}]
    Theorem \ref{thm-links} applies to the chosen diagram of $L$ with the same normalized Alexander polynomial $F_L$ and coefficients $\gamma_i$.
    It gives positivity and palindromicity, the inequalities \eqref{eq-link-quadratic} and \eqref{eq-link-power}, and the equality criteria for log-concavity and power concavity.
\end{proof}

Theorem \ref{thm-links} refines the ordinary log-concavity theorem of \cite[Theorem 1.2]{HMV24} by giving the quadratic inequality depending on $s$, the power-concavity exponent, and the rigidity of equality.
For $s=3$, \eqref{eq-link-power} is the additive concavity inequality $2\gamma_i\geq\gamma_{i-1}+\gamma_{i+1}$.
For $s\geq4$, the exponent $1/(s-2)$ remains positive and depends explicitly on the Seifert-circle count of the chosen diagram.

\printbibliography

\end{document}